\documentclass[11pt]{amsart}

\usepackage{amsmath, amssymb, amsthm}

\usepackage{mathrsfs}

\usepackage{tikz-cd}

\numberwithin{equation}{section}

\newtheorem{theorem}{Theorem}[section]
\newtheorem*{theorem*}{Theorem}
\newtheorem{lemma}[theorem]{Lemma}
\newtheorem{proposition}[theorem]{Proposition}
\newtheorem{corollary}[theorem]{Corollary}

\theoremstyle{definition}
\newtheorem{definition}[theorem]{Definition}
\newtheorem{example}[theorem]{Example}

\theoremstyle{remark}
\newtheorem{remark}[theorem]{Remark}

\newcommand{\sw}[2]{{#1}_{(#2)}}
\newcommand{\swy}[2]{{#1}^{(#2)}}

\newcommand{\HM}[2]{{}^{#1}\mathfrak{M}_{#2}}
\newcommand{\twg}[2]{{}^{[#2]}\! {#1}}
\newcommand{\Coker}{\operatorname{Coker}}
\newcommand{\Ker}{\operatorname{Ker}}
\newcommand{\Image}{\operatorname{Im}}
\newcommand{\M}{\mathfrak{M}}

\newcommand{\YD}{\mathcal{YD}}
\newcommand{\coinv}[1]{{#1}^\mathrm{coinv}}
\newcommand{\mapsfrom}{\;\raisebox{5pt}{\rotatebox{180}{$\mapsto$}} \;}
\newcommand{\corad}[1]{\mathrm{corad}({#1})}
\newcommand{\soc}[1]{\mathrm{soc}({#1})}

\newcommand{\id}{\mathrm{id}}

\newcommand{\ev}{\mathrm{ev}}
\newcommand{\coev}{\mathrm{coev}}

\newcommand{\gr}{\operatorname{gr}}

\title[Freeness over right $H$-simple left $H$-comodule algebras]
{Freeness and divisibility for right $H$-simple left $H$-comodule algebras over a pointed Hopf algebra $H$}
\author{Daisuke Nakamura}

\address{
Graduate School of Science and Engineering,
Okayama University of Science,
1-1 Ridai-cho, Kita-ku Okayama-shi, Okayama 700-0005, Japan.
}

\email{r25ndk3od@ous.jp}


\begin{document}

\begin{abstract}
Let $H$ be a pointed Hopf algebra and let $A$ be a right $H$-simple left $H$-comodule algebra.
We show that every relative $(H,A)$-Hopf module is free as an $A$-module
and that this freeness characterizes the class of pointed Hopf algebras.
We give a criterion for the category of relative $(H,A)$-Hopf modules to be semisimple.
We also show that $A$ can be embedded into a left $H$-comodule algebra of a specific form when $H$ and $A$ are $\mathbb{N}_0$-graded.
As a consequence, we prove that if $H$ is finite-dimensional and $A^{\mathrm{co} H}=\Bbbk$,
then $A$ is finite-dimensional and $\dim A$ divides $\dim H$.
\end{abstract}

\maketitle

\section{Introduction}

Freeness, flatness, faithful flatness and projectivity of a Hopf algebra over its Hopf subalgebras and one-sided coideal subalgebras have been studied by many researchers,
with relative Hopf modules playing an important role in these studies.
Radford proved that a pointed Hopf algebra $H$ is free over any Hopf subalgebra $K$~\cite{R77}.
Takeuchi extended this result to relative Hopf modules, showing that every relative $(H,K)$-Hopf module is free as a $K$-module when $H$ is a pointed Hopf algebra and $K$ is a Hopf subalgebra~\cite{T79}.
Masuoka further extended this result to the case where $K$ is a one-sided coideal subalgebra of a pointed Hopf algebra $H$ such that $G(H)\cap K$ is a group~\cite{Ma91},
where $G(H)$ is the group of group-like elements of $H$. 
Without assuming pointedness, Nichols and Zoeller proved that every relative $(H,K)$-Hopf module is free as a $K$-module when $H$ is finite-dimensional and $K$ is a Hopf subalgebra~\cite{NZ89},
and Skryabin later extended this result to one-sided coideal subalgebras~\cite{Sk07}.
Note that Skryabin established projectivity and, under additional assumptions, freeness of relative Hopf modules over $H$-simple $H$-comodule algebras assuming certain weak finiteness conditions.
However, in the infinite-dimensional nonpointed setting, $H$ needs not be flat over $K$ even if the Hopf subalgebra $K$ is finite-dimensional~\cite{Sk25}.
One of our purposes is to explain why the above classical freeness results of Takeuchi and Masuoka hold even if $H$ is infinite-dimensional
in terms of the pointedness and one-sided $H$-simplicity. 
This viewpoint leads to a characterization of the class of pointed Hopf algebras, which is one of our main results, as follows.
\begin{theorem*}[$=$ Theorem~\ref{freencharpointed}]
Let $H$ be a Hopf algebra. The following are equivalent:
\begin{itemize}
\item [\textup{(i)}] $H$ is pointed.
\item [\textup{(ii)}] For any left $H$-simple left $H$-comodule algebra $A$, every object of ${}^H_A\M$ is free as an $A$-module.
\item [\textup{(iii)}] For any right $H$-simple left $H$-comodule algebra $A$, every object of $\HM{H}{A}$ is free as an $A$-module.
\end{itemize}
Here ${}^H_A\M$ (resp. $\HM{H}{A}$) denotes the category of left-left (resp. left-right) relative $(H,A)$-Hopf modules.
\end{theorem*}
In this paper, we mainly use the right-handed version of one-sided $H$-simplicity of left $H$-comodule algebras.
Note that left $H$-simplicity and right $H$-simplicity coincide when $H$ and $A$ are finite-dimensional~\cite[Lemma~2.12]{NSS25}.
We also show this holds when $H$ is pointed, without assuming either $H$ or $A$ is finite-dimensional.

The notion of a right $H$-simple left $H$-comodule algebra can be regarded as a generalization of the notion of a $G$-graded division algebra.
It is well known that every $G$-graded module over a $G$-graded division algebra $A$ is $G$-graded free.
This freeness is stronger than just freeness. The $G$-graded freeness implies that every object of ${\HM{\Bbbk G}{A}}$ is free as an $A$-module and that ${\HM{\Bbbk G}{A}}$ is a semisimple category.
Motivated by this, we give a criterion for ${\HM{H}{A}}$ to be semisimple when $A$ is a right $H$-simple left $H$-comodule algebra over a pointed Hopf algebra $H$;
see also earlier work by Caenepeel and Gu{\'e}d{\'e}non~\cite{CG4} on the semisimplicity of the category of relative Hopf modules.

From another point of view,
the class of right $H$-simple left $H$-comodule algebras has a categorical interpretation.
{\emergencystretch=1em
Andruskiewitsch and Mombelli showed that indecomposable exact module categories over the category $\mathrm{Rep}(H)$ of finite-dimensional representations of a finite-dimensional Hopf algebra $H$ correspond to finite-dimensional right $H$-simple left $H$-comodule algebras~\cite{AM7};
see also~\cite[Section~3.7]{NSS25} for an exposition.
Moreover, the dimension of a right $H$-simple left $H$-comodule algebra over a pointed Hopf algebra $H$ is an $H$-Morita invariant~\cite[Theorem~3.11]{NSS25}.
Mombelli proved that the dimension of a right $H$-simple left $H$-comodule algebra with trivial coinvariants divides the dimension of a finite-dimensional pointed Hopf algebra $H$ such that $G(H)$ is cyclic~\cite[Lemma~7.7]{M10}.
As an application of our freeness theorem,
we show that this divisibility holds for any finite-dimensional pointed Hopf algebra.
\par
}
\begin{theorem*}[$=$ Theorem~\ref{prHLag}]
Let $H$ be a finite-dimensional pointed Hopf algebra.
For every right $H$-simple left $H$-comodule algebra $A$ with $A^{\mathrm{co}H} = \Bbbk$,
$\dim A <\infty$ and $\dim A$ divides $\dim H$.
\end{theorem*}
Note that this theorem is not an immediate corollary of the freeness theorem because $H$ is not necessarily an object of $\HM{H}{A}$.
So we also show that $\HM{H}{A}$ has a certain object $X$ such that $\dim X$ divides $\dim H$ under the above assumptions when $H$ and $A$ are $\mathbb{N}_0$-graded,
and this result applies to the associated graded algebras with respect to a suitable filtration.

\subsection*{Organization of the paper}
The paper is organized as follows.
In Section~\ref{notation}, we review basic definitions and notation concerning pointed Hopf algebras, comodule algebras, relative Hopf modules, left and right $H$-simplicity and Yetter-Drinfeld modules.
In Section~\ref{Cfilteredobject}, we recall the notion of $\mathcal{C}$-filtered objects, which is a main tool to obtain our results.
We apply this notion to the study of freeness, flatness, faithful flatness, projectivity and semisimplicity of relative Hopf modules.
In Section~\ref{freenesssemisimplicity},
we show freeness, classify simple objects of ${\HM{H}{A}}$ and give a criterion for ${\HM{H}{A}}$ to be semisimple when $H$ is pointed and $A$ is right $H$-simple.
We also show that left and right $H$-simplicity coincide in the pointed case.
In addition, we describe how the classical freeness results of Takeuchi and Masuoka follow from our result.
We further observe that a right $H$-simple $H$-comodule algebra is free over any finite-dimensional comodule subalgebra when $H$ is pointed.
In Section~\ref{charapointedbyfreeness}, we show that our freeness theorem characterizes the class of pointed Hopf algebras.
In Section~\ref{rightHsimplecomodalg},
we begin with the well-known structure theorem for right $\Bbbk G$-simple $\Bbbk G$-comodule algebras
and end up showing that a right $H$-simple comodule algebra $A$ with trivial coinvariants over a pointed Hopf algebra $H$ can be embedded into a certain left $H$-comodule algebra, constructed from $H$, when $H$ and $A$ are $\mathbb{N}_0$-graded.
In the final Section~\ref{Lagrange}, we prove a Lagrange-type theorem by combining the freeness theorem with the graded structure theorem.
In Appendix~\ref{classidimH/Gprime}, we give a further divisibility result and in the case where $\dim H/\dim\corad{H}$ is prime, explain that the possible structures of right $H$-simple $H$-comodule algebras are restricted.

\subsection*{Acknowledgments}
The author thanks Akira Masuoka for pointing out a more general perspective on Proposition~\ref{annmor} and Lemma~\ref{hassimpAg}, as discussed in Remark~\ref{embedMgDA}, and for suggesting an alternative proof of Theorem~\ref{mainresult}.
The author thanks Kenichi Shimizu for pointing out that the biproduct construction in Section~\ref{adjforcomalg} can be understood in a more general framework,
which led to a more conceptual  proof of Proposition~\ref{adcomodalg}, and for helpful comments.
The author is grateful to Taiki Shibata for carefully reading an earlier version of this paper and for many helpful suggestions.
Some of the results in this paper were announced at the 57th Symposium on Ring Theory and Representation Theory held on September 8--12, 2025 at Okayama University of Science, Okayama, Japan
(see \cite{N26}).

\section{Notation and conventions}\label{notation}
Throughout this paper, we work over a field $\Bbbk$.
We denote the tensor product over $\Bbbk$ simply by $\otimes$.

For a coalgebra $C$ and an algebra $A$,
we denote by ${}^C\M$, $\M^C$, ${}_A\M$ and $\M_A$ the categories of left $C$-comodules, right $C$-comodules, left $A$-modules and right $A$-modules, respectively. 
In this paper, by a $C$-comodule, we mean a left $C$-comodule.

Let $H$ be a Hopf algebra with comultiplication $\Delta$ and counit $\varepsilon$
and let $M$ be an $H$-comodule with $H$-coaction $\rho_M$.
We denote by $\soc{M}$ the \emph{socle} of $M$, which is the sum of all simple $H$-subcomodules.
In particular, the socle of $H$ coincides with its \emph{coradical} $\corad{H}$, which is the sum of all simple subcoalgebras (see, for example \cite[Theorem~3.4.2]{Radtext}).

We record the following standard fact.
\begin{lemma}\label{socorad}
The following holds:
\[\soc{M} = \rho_M^{-1}(\corad{H}\otimes M).\]
\end{lemma}
\begin{proof}
By \cite[Theorem~3.2.11(d)]{Radtext}, which says for a simple $H$-comodule $N$, there exists a simple subcoalgebra $S$ of $H$ such that $\rho_M(N)\subset S\otimes N$,
we have $\soc{M} \subset \rho_M^{-1}(\corad{H}\otimes M)$.

We show the converse. Put $N:=\rho_M^{-1}(\corad{H}\otimes M)$.
Since $\rho_M$ is an $H$-comodule map, $N$ is an $H$-subcomodule of $M$,
and hence $\rho_M(N)\subset (\corad{H}\otimes M) \cap (H\otimes N) = \corad{H}\otimes N$.
Thus, $N$ is a $\corad{H}$-comodule.
Hence, $N$ is the sum of its simple $\corad{H}$-subcomodules (see \cite[Theorem~3.4.10]{Radtext}).
By the definition of $N$,
any $H$-subcomodule of $N$ is a $\corad{H}$-subcomodule, which implies that any simple $\corad{H}$-subcomodule of $N$ is a simple $H$-subcomodule.
Therefore, $N\subset \soc{M}$.
\end{proof}
We denote by $M^{\mathrm{co}H} := \{m\in M \mid \rho_M(m) = 1\otimes m\}$ the \emph{coinvariant space} of $M$.
A Hopf algebra is called \emph{pointed} if any simple subcoalgebra of it is one-dimensional.
In particular, if $H$ is pointed, then $\corad{H}$ is the group algebra of the \emph{group-like elements} $G(H) := \{g\in H \mid \Delta(g)=g\otimes g\text{, }\varepsilon(g)=1\}$ of $H$.

{
\emergencystretch=1em
We record the following elementary lemma about comodules over a pointed Hopf algebra.
\par
}
\begin{lemma}\label{comodoverpointed}
    Suppose $H$ is pointed.
    Let $0\neq M$ be an $H$-comodule.
    Then there exist $g\in G(H)$ and $0\neq m\in M$ such that $\rho_M(m) = g\otimes m$. 
\end{lemma}
\begin{proof}
By the fundamental theorem of comodules, which says a nonzero comodule has a nonzero finite-dimensional subcomodule, $M$ has a simple $H$-subcomodule $N$.
By \cite[Theorem~3.2.11(d)]{Radtext}, there exists a simple subcoalgebra $S$ of $H$ such that $\rho_M(N)\subset S\otimes N$.
Since $H$ is pointed, we can write $S=\Bbbk g$ for some $g\in G(H)$.
Thus $\rho_M(N) \subset \Bbbk g \otimes N$, which implies there exist $0 \neq x\in N$ and $x'\in N$ such that $\rho_M(x)=g\otimes x'$.
Since $1\otimes x =(\varepsilon \otimes \mathrm{id})\circ \rho_M(x) = \varepsilon(g)\otimes x' = 1\otimes x'$, we have $\rho_M(x)=g\otimes x$.
\end{proof}

A \emph{left $H$-comodule algebra} is an algebra which is a left $H$-comodule whose $H$-coaction is an algebra map.
In this paper, by an $H$-comodule algebra, we mean a left $H$-comodule algebra.
An \emph{$H$-comodule algebra map} is an $H$-comodule map which is an algebra map.
We denote by ${}^H\! \mathrm{Alg}$ the category of $H$-comodule algebras.

For an $H$-comodule algebra $A$ with $H$-coaction $\rho_A$,
we denote by ${}^H_A\M$ (resp. $\HM{H}{A}$) the category of \emph{left-left} (resp. \emph{left-right}) \emph{relative $(H,A)$-Hopf modules}.
An object $M\in {}^{H}_{A}\M$ is an object of both ${}^H\M$ and ${}_A\M$ such that
$\rho_M(a.m)=\sw{a}{-1}\sw{m}{-1}\otimes \sw{a}{0}.\sw{m}{0}$ for $a\in A$ and $m\in M$,
where we use the Sweedler notation $\rho_A(a)=\sw{a}{-1}\otimes \sw{a}{0}$ and $\rho_M(m)=\sw{m}{-1}\otimes \sw{m}{0}$.
A morphism in ${}^{H}_{A}\M$ is a left $H$-comodule map which is also a left $A$-module map.
Similarly, an object $M\in \HM{H}{A}$ is an object of both ${}^H\M$ and $\M_A$ such that
$\rho_M(m.a)=\sw{m}{-1}\sw{a}{-1}\otimes \sw{m}{0}.\sw{a}{0}$ for $a\in A$ and $m\in M$.
A morphism in $\HM{H}{A}$ is a left $H$-comodule map which is also a right $A$-module map.

\begin{definition}
Let $A$ be an $H$-comodule algebra.
\begin{itemize}
        \item An \emph{$H$-comodule left ideal} (resp. \emph{$H$-comodule right ideal}) of $A$ is a left ideal (resp. right ideal) which is an $H$-subcomodule of $A$.
        \item $A$ is called \emph{left $H$-simple} (resp. \emph{right $H$-simple}) if it is nonzero and any $H$-comodule left ideal (resp. right ideal) of $A$ is trivial. 
\end{itemize}
\end{definition}
In particular, $A$ is left $H$-simple (resp. right $H$-simple) if and only if $A$ is a simple object in ${}^{H}_{A}\M$ (resp. $\HM{H}{A}$).

\begin{example}
For a group $F$, a map $\psi \colon F\times F \to \Bbbk^\times$ is said to be a \emph{$2$-cocycle} of $F$ if it satisfies $\psi(a,b)\psi(ab,c)=\psi(b,c)\psi(a,bc)$ and $\psi(a,e)=1=\psi(e,a)$ for $a,b,c\in F$.
For a group $G$, a subgroup $F$ of $G$ and a $2$-cocycle $\psi$ of $F$, we can define the $\Bbbk G$-comodule algebra ${}_\psi \Bbbk F$ as follows:
Its underlying vector space is $\Bbbk F$, the multiplication is given by $g * h := \psi(g,h)gh$ for $g,h\in F$, where $gh$ denotes the product in $F$
and the $\Bbbk G$-coaction is given by $\rho(g)=g\otimes g$ for $g\in F$.
It is well known that a $\Bbbk G$-comodule algebra is of the form ${}_\psi\Bbbk F$ if and only if it is right $\Bbbk G$-simple and has trivial coinvariants (see \cite{BZS8}).
We also give another proof of this fact later in Proposition~\ref{rightKGsimple} using our criterion for right $H$-simplicity.
\end{example}
We denote by ${}^H_H\mathcal{YD}$ the category of \emph{left-left Yetter-Drinfeld modules} over $H$.
An object $M\in {}^H_H\YD$ is an object of both ${}_H\M$ and ${}^H\M$ such that
$\sw{h}{1}\sw{m}{-1}\otimes \sw{h}{2}.\sw{m}{0} = \sw{(\sw{h}{1}.m)}{-1}\sw{h}{2}\otimes \sw{(\sw{h}{1}.m)}{0}$ for $h\in H, m\in M$, 
where we use the Sweedler notation $\Delta(h)=\sw{h}{1}\otimes \sw{h}{2}$.
A morphism in ${}^{H}_{H}\YD$ is a left $H$-module map which is also a left $H$-comodule map.

\section{$\mathcal{C}$-filtered objects}\label{Cfilteredobject}

\subsection{$\mathcal{C}$-filtered modules}
Let $A$ be an algebra and let $\mathcal{C}$ be a class of objects in $\M_A$.
We recall the standard terminology of $\mathcal{C}$-filtered modules (see, for example \cite[Definition~6.1]{GTtext}).
We then collect some consequences that will be used in the study of relative Hopf modules.

\begin{definition}\label{def_fobjformodule}
An object $M\in \M_A$ is called \emph{$\mathcal{C}$-filtered}
if there exists an ordinal number $\lambda$ and $A$-submodules $(M_\alpha)_{\alpha\leq \lambda}$ of $M$ such that
\begin{itemize}
    \item $M_0 = 0$ and $M_\lambda =M$.
    \item for each $\alpha\leq \beta \leq \lambda$, $M_\alpha \subset M_\beta$.
    \item for each $\alpha < \lambda$, $M_{\alpha+1}/M_\alpha$ is isomorphic to an object of $\mathcal{C}$.
    \item for each limit ordinal $\alpha\leq \lambda$, $M_\alpha = \bigcup_{\alpha'<\alpha} M_{\alpha'}$.
\end{itemize}
\end{definition}

\begin{proposition}\label{trexflat}
Let $M\in \M_A$ be $\mathcal{C}$-filtered.
If every object of $\mathcal{C}$ is a flat $A$-module, then $M$ is a flat $A$-module.
Moreover, if there exists an ordinal $\beta$ such that $M_{\beta+1}/M_\beta$ is faithfully flat, then $M$ is also faithfully flat.
\end{proposition}
\begin{proof}
Suppose $M$ is $\mathcal{C}$-filtered by $(M_\alpha)_{\alpha\leq \lambda}$.
To prove the first statement, we show that $M_\alpha$ is flat for all $\alpha\leq\lambda$ by transfinite induction.
If $\alpha = 0$, then the statement follows trivially.
Assume that the statement holds for all ordinals less than $\alpha$.

If $\alpha = \alpha'+1$ is a successor ordinal, then we have the following exact sequence
\[0\longrightarrow M_{\alpha'}\longrightarrow M_{\alpha'+1}\longrightarrow M_{\alpha'+1}/M_{\alpha'}\longrightarrow 0\]
in $\M_A$. Thus, we have an exact sequence
\[\mathrm{Tor}^A_1(M_{\alpha'},X) \longrightarrow \mathrm{Tor}^A_1(M_{\alpha'+1},X) \longrightarrow  \mathrm{Tor}^A_1(M_{\alpha'+1}/M_{\alpha'},X) \]
for $X\in {}_A \M$. By the induction hypothesis and the assumption, $M_{\alpha'}$ and $M_{\alpha'+1}/M_{\alpha'}$ are flat, and hence $\mathrm{Tor}^A_1(M_{\alpha'},X),  \mathrm{Tor}^A_1(M_{\alpha'+1}/M_{\alpha'},X) = 0$.
Thus, $\mathrm{Tor}^A_1(M_{\alpha'+1},X) = 0$ for $X\in {}_A\M$, so $M_\alpha = M_{\alpha'+1}$ is flat.

If $\alpha$ is  a limit ordinal, then $M_\alpha = \bigcup_{\alpha'<\alpha} M_{\alpha'}$ is an inductive limit of $(M_{\alpha'})_{\alpha'<\alpha}$, which are flat by the induction hypothesis.
Therefore, $M_{\alpha}$ is flat since an inductive limit of flat modules is flat.

Next, we show the latter statement.
Assume there exists an ordinal $\beta<\lambda$ such that $M_{\beta+1}/M_\beta$ is faithfully flat.
We show that $M$ is faithfully flat.
Let $N$ be a left $A$-module such that $M\otimes_A N =0$.
Let $Q := M_{\beta+1}/M_\beta$, $E := M/M_\beta$, $T := M/M_{\beta+1}$.
We get the following exact sequence
\[0\longrightarrow Q \longrightarrow E \longrightarrow T \longrightarrow 0.\]
Then
\[\mathrm{Tor}^A_1(T,N) \longrightarrow Q \otimes_A N \longrightarrow E \otimes_A N \longrightarrow T \otimes_A N\longrightarrow 0.\]
We show that $T$ is flat.
Let $T_\gamma := M_\gamma/M_{\beta+1}$ for $\beta+1 \leq \gamma\leq \lambda$.
Then $T$ is $\mathcal{C}$-filtered by $(T_\gamma)_\gamma$.
Indeed, $T_{\beta+1} = M_{\beta+1}/M_{\beta+1} = 0$, $T_{\lambda} = M_\lambda / M_{\beta+1} = M/M_{\beta+1} = T$
and $T_{\gamma+1}/T_{\gamma} = (M_{\gamma+1}/M_{\beta+1})/(M_{\gamma}/M_{\beta+1}) \cong M_{\gamma+1}/M_{\gamma}$.
Therefore, we have that $T$ is flat by the first statement. Hence, $\mathrm{Tor}^A_1(T,N) = 0$,
and thus $Q \otimes_A N \longrightarrow E \otimes_A N$ is injective.
Moreover, by the right exactness, $0=M \otimes_A N \longrightarrow E \otimes_A N$ is surjective, and hence $Q\otimes_A N \longrightarrow E \otimes_A N = 0$.
Since $Q=M_{\beta+1}/M_\beta$ is faithfully flat, we have $N=0$.
\end{proof}

\begin{proposition}\label{trexprojective}
Let $M\in \M_A$ be $\mathcal{C}$-filtered.
If every object of $\mathcal{C}$ is a projective $A$-module, then $M$ is a direct sum of objects of $\mathcal{C}$ as an $A$-module.
In particular, $M$ is projective.
\end{proposition}
\begin{proof}
Suppose $M$ is $\mathcal{C}$-filtered by $(M_\alpha)_{\alpha\leq \lambda}$.
For each $\alpha<\lambda$, we have the following exact sequence
\[0\longrightarrow M_{\alpha}\longrightarrow M_{\alpha+1}\longrightarrow M_{\alpha+1}/M_{\alpha}\longrightarrow 0\]
in $\M_A$. By the assumption, $M_{\alpha+1}/M_{\alpha}$ is projective,
and hence we can choose an $A$-submodule $P_\alpha$ of $M_{\alpha+1}$ such that $M_{\alpha+1} = M_{\alpha}\oplus P_\alpha$ and $P_\alpha \cong M_{\alpha+1}/M_{\alpha}$ as an $A$-module for each $\alpha<\lambda$.

We now show that for each $\alpha\leq \lambda$, $M_\alpha = \bigoplus_{\alpha'<\alpha} P_{\alpha'}$ by transfinite induction.
Assume that the statement holds for all ordinals less than $\alpha$.

If $\alpha = \alpha'+1$ is a successor ordinal,
\[M_{\alpha'+1} = M_{\alpha'} \oplus P_{\alpha'} = (\bigoplus_{\alpha''<\alpha'} P_{\alpha''})\oplus P_{\alpha'} = \bigoplus_{\alpha''\leq\alpha'} P_{\alpha''} = \bigoplus_{\alpha''<\alpha'+1} P_{\alpha''}.\]

If $\alpha$ is a limit ordinal,
$M_\alpha = \bigcup_{\alpha'<\alpha} {M_{\alpha'}} = \bigcup_{\alpha'<\alpha} \bigoplus_{\gamma<\alpha'} P_{\gamma} = \sum_{\alpha'<\alpha} P_{\alpha'}$.
We show this is a direct sum. 
Suppose $x_{\gamma_1} + \cdots + x_{\gamma_n} = 0 \quad (n\in \mathbb{N},  x_{\gamma_i}\in P_{\gamma_i})$ where $\gamma_1,\ldots,\gamma_n<\alpha$ are distinct ordinals.
Put $\gamma := \max\{\gamma_1,\ldots,\gamma_n\} +1$, which is less than $\alpha$ since $\alpha$ is a limit ordinal.
By the induction hypothesis, $\sum_{\gamma'<\gamma} P_{\gamma'}$ is a direct sum, which implies each $x_{\gamma_i}$ is zero.

Therefore, we have $M = M_\lambda = \bigoplus_{\alpha<\lambda} P_{\alpha}$.
\end{proof}

\begin{corollary}\label{trexfree}
Let $M\in \M_A$ be $\mathcal{C}$-filtered.
If every object of $\mathcal{C}$ is a free $A$-module, then $M$ is a free $A$-module.
\end{corollary}
\begin{proof}
By Proposition~\ref{trexprojective}, $M$ is a direct sum of free modules.
\end{proof}

\subsection{$\mathcal{C}$-filtered objects in an abelian category}
Let $\mathscr{A}$ be an abelian category and let $\mathcal{C}$ be a class of objects of $\mathscr{A}$.
We recall the terminology of $\mathcal{C}$-filtered objects in an abelian category (see, for example \cite[Definition~1.3]{S13}).
\begin{definition}\label{DefCfiltobj}
An object $M\in \mathscr{A}$ is called \emph{$\mathcal{C}$-filtered}
if there exist an ordinal $\lambda$ and a direct system $(M_\alpha,i_{\alpha,\beta} \colon M_\alpha \longrightarrow M_{\beta})_{\alpha\leq \beta \leq \lambda}$ in $\mathscr{A}$
such that
\begin{itemize}
    \item $M_0 = 0$ and $M_\lambda =M$.
    \item for each $\alpha\leq \beta\leq \lambda$, $i_{\alpha,\beta} \colon M_\alpha \longrightarrow M_{\beta}$ is a monomorphism.
    \item for each $\alpha < \lambda$, $\Coker (i_{\alpha,\alpha+1})$ is isomorphic to an object of $\mathcal{C}$.
    \item for each limit ordinal $\gamma\leq \lambda$, $(M_\gamma, (i_{\alpha, \gamma}\colon M_\alpha \longrightarrow M_\gamma)_{\alpha<\gamma})$ is the colimit
    of the direct system $(M_\alpha,i_{\alpha,\beta})_{\alpha\leq\beta<\gamma}$.
\end{itemize}
\end{definition}

Let $A$ be an algebra.

\begin{lemma}\label{fpf}
Let $M\in\mathscr{A}$ be $\mathcal{C}$-filtered in the sense of Definition~\ref{DefCfiltobj} and let $U \colon \mathscr{A} \to \M_A$ be an exact functor which preserves filtered colimits.
Then $U(M)$ is $U(\mathcal{C})$-filtered in the sense of Definition~\ref{def_fobjformodule}.
\end{lemma}
\begin{proof}
Suppose $M$ is $\mathcal{C}$-filtered by $(M_\alpha,i_{\alpha,\beta} \colon M_\alpha \longrightarrow M_{\beta})_{\alpha\leq \beta \leq \lambda}$.
Put $N := U(M)$ and $N_\alpha := \Image ( U(i_{\alpha,\lambda}))$ for $\alpha \leq \lambda$.
For $\alpha\leq \beta\leq \lambda$,
since $(i_{\gamma,\mu})_{\gamma\leq \mu\leq \lambda }$ is a direct system, $i_{\alpha,\lambda} = i_{\beta,\lambda}\circ i_{\alpha,\beta}$.
Hence,
$N_\alpha = \Image (U(i_{\alpha,\lambda})) = \Image (U(i_{\beta,\lambda})\circ U(i_{\alpha,\beta})) \subset \Image (U(i_{\beta,\lambda})) = N_\beta.$
Since an exact functor preserves monomorphisms, we have $N_\alpha \cong U(M_\alpha)$.
Thus, $N_0 \cong U(M_0) = U(0) \cong 0$ since an exact functor preserves the zero object. Hence, $N_0=0$.
We have
$N_\lambda = \Image (U(i_{\lambda,\lambda})) = \Image (U(\mathrm{id}_{M_\lambda})) = \Image (\mathrm{id}_{U(M_\lambda)})= U(M_\lambda) = U(M)$.
Since an exact functor preserves cokernels,
$U(\Coker(i_{\alpha,\alpha+1})) \cong \Coker (U(i_{\alpha,\alpha+1})) \cong N_{\alpha+1}/N_{\alpha}$ and is isomorphic to an object of $U(\mathcal{C})$ for $\alpha< \lambda$.
If $\alpha$ is a limit ordinal, then $N_\alpha\cong U(M_\alpha)$ is the inductive limit of $(N_0\subset \ldots \subset N_{\alpha'})_{\alpha'<\alpha}$ since $U$ preserves the filtered colimit.
On the other hand, $\bigcup_{\alpha'<\alpha} N_{\alpha'}$ is also the inductive limit of $(N_0\subset \ldots \subset N_{\alpha'})_{\alpha'<\alpha}$.
By the universal property, the inclusion $\bigcup_{\alpha'<\alpha} N_{\alpha'} \hookrightarrow N_\alpha$ is an isomorphism,
which implies $N_\alpha = \bigcup_{\alpha'<\alpha} N_{\alpha'}$.
\end{proof}

\begin{proposition}\label{ifU}
Let $M\in\mathscr{A}$ be $\mathcal{C}$-filtered and $U \colon \mathscr{A} \to  \M_A$ be an exact functor which preserves filtered colimits.
Then the following statements hold:
\begin{enumerate}
    \item If every object of $U(\mathcal{C})$ is flat, then $U(M)$ is flat.
    \item If every object of $U(\mathcal{C})$ is faithfully flat and $M\neq 0$, then $U(M)$ is faithfully flat.
    \item If every object of $U(\mathcal{C})$ is projective, then $U(M)$ is projective. More precisely, $U(M)$ is a direct sum of objects of $U(\mathcal{C})$.
    \item If every object of $U(\mathcal{C})$ is free, then $U(M)$ is free.
\end{enumerate}
\end{proposition}
\begin{proof}
By Propositions~\ref{trexflat}, \ref{trexprojective}, Corollary~\ref{trexfree} and Lemma~\ref{fpf}.
\end{proof}

\subsection{$\mathcal{C}$-filtered relative Hopf modules}
We now apply the preceding results to relative Hopf modules.
Let $H$ be a Hopf algebra and let $A$ be an $H$-comodule algebra.
Let $\mathcal{C}$ be a class of objects of $\HM{H}{A}$.
\begin{definition}
We say that an object $M\in \HM{H}{A}$ is \emph{$\mathcal{C}$-semiartinian}
if for any relative $(H,A)$-Hopf submodule $N\neq M$ of $M$,
$M/N$ has a nonzero relative $(H,A)$-Hopf submodule which is isomorphic to an object of $\mathcal{C}$.
\end{definition}
Note that every object in $\HM{H}{A}$ is $\mathcal{C}$-semiartinian if and only if every nonzero object in $\HM{H}{A}$ has a nonzero subobject which is isomorphic to an object of $\mathcal{C}$.

\begin{proposition}\label{altofilt}
If $M\in \HM{H}{A}$ is $\mathcal{C}$-semiartinian, then $M$ is $\mathcal{C}$-filtered.
\end{proposition}
\begin{proof}
Let $\mathcal{S} := \{N\mid \text{$N$ is a relative $(H,A)$-Hopf submodule of $M$}\}$, which is a set as $\mathcal{S}$ is a subset of the power set of $M$.
By Hartogs' lemma, there exists an ordinal $\mu$ such that there is no injection from $\mu$ to $\mathcal{S}$.
Since $M$ is $\mathcal{C}$-semiartinian, by the correspondence theorem,
for each $N\in \mathcal{S}\setminus \{M\}$, we can choose $s(N)\in \mathcal{S}$ such that $N\subsetneq s(N)$ and $s(N)/N$ is isomorphic to an object of $\mathcal{C}$.
For each $\alpha\leq \mu$, we define $M_\alpha$ as follows:
\begin{itemize}
    \item $M_0 := 0$.
    \item If $\alpha=\alpha'+1$ is a successor ordinal,
    \[M_{\alpha} := \begin{cases}
    s(M_{\alpha'})&  \text{if }M_{\alpha'} \in \mathcal{S}\setminus \{M\},\\
    M& \text{otherwise}.
    \end{cases}\]
    \item If $\alpha$ is a limit ordinal, $M_\alpha := \bigcup_{\alpha'<\alpha}M_{\alpha'}$.
\end{itemize}
We show that 
\[M_\beta \in \mathcal{S} \text{ and for each } \alpha \leq \beta , M_\alpha\subset M_\beta\]
by transfinite induction on $\beta\leq \mu$.
If $\beta = 0$, then the statement holds since $M_0 = 0 \in \mathcal{S}$.
Assume that the statement holds for all ordinals less than $\beta$.
If $\beta$ is a successor ordinal, then it is clear from the definition of $M_\beta$ and the induction hypothesis that the statement holds.
If $\beta$ is a limit ordinal, then for each $\alpha < \beta$, $M_\alpha \subset \bigcup_{\beta'<\beta} M_{\beta'} = M_\beta$.
Moreover, by the induction hypothesis, $(M_{\beta'})_{\beta'<\beta}$ is an ascending chain of relative $(H,A)$-Hopf submodules of $M$,
and hence its union is again a relative $(H,A)$-Hopf submodule of $M$. Thus, $M_\beta \in \mathcal{S}$.
Thus, by transfinite induction, the statement holds.

We can consider a map $f\colon \mu \to \mathcal{S};\;\alpha \mapsto M_\alpha$.
Assume that for any $\beta < \mu$, $M_\beta \neq M$.
Let $\alpha'<\alpha < \mu$.
Then $\alpha' +1 \leq \alpha$ and $M_{\alpha'+1} \subset M_{\alpha}$.
By the assumption, $M_{\alpha'}\in \mathcal{S}\setminus \{M\}$, and hence $M_{\alpha'+1} = s(M_{\alpha'})\supsetneq M_{\alpha'}$.
Hence, we have $M_{\alpha'} \subsetneq M_{\alpha}$.
This implies $f$ is an injection, which contradicts the choice of $\mu$.
Therefore, $\emptyset \neq \{\mu'< \mu \mid M_{\mu'} = M\}$,
and hence we can take the minimum element $\lambda$ of this set.
Then the relative $(H,A)$-Hopf submodules $(M_\alpha)_{\alpha\leq \lambda}$ of $M$ 
satisfy the following conditions.
\begin{itemize}
    \item $M_0 = 0$ and $M_\lambda =M$.
    \item for each $\alpha\leq \beta \leq \lambda$, $M_\alpha \subset M_\beta$.
    \item for each $\alpha < \lambda$, $M_{\alpha+1}/M_\alpha$ is isomorphic to an object of $\mathcal{C}$.
    \item for each limit ordinal $\alpha\leq \lambda$, $M_\alpha = \bigcup_{\alpha'<\alpha} M_{\alpha'}$.
\end{itemize}
This means $M$ is a $\mathcal{C}$-filtered object in $\HM{H}{A}$.
\end{proof}

\begin{remark}
The notion of a $\mathcal{C}$-semiartinian object can be defined in an abelian category $\mathscr{A}$ other than $\HM{H}{A}$.
One may obtain Proposition~\ref{altofilt} in the same way
for a $\mathcal{C}$-semiartinian object $M\in \mathscr{A}$ such that $\mathrm{Sub}(M)$ is a set
and for any ascending chain $(M_\alpha)_\alpha$ of subobjects of $M$,
there exists a colimit $\operatorname{colim}M_\alpha$ in $\mathscr{A}$ and
the induced morphism $\operatorname{colim}M_\alpha \to M$ is a monomorphism,
where $\mathrm{Sub}(M)$ denotes the class of all subobjects of $M$.
\end{remark}

\begin{proposition}\label{semiarttoM}
Let $M\in \HM{H}{A}$ be $\mathcal{C}$-semiartinian.
Then the following statements hold:
\begin{enumerate}
    \item If every object of $\mathcal{C}$ is flat as an $A$-module, then $M$ is flat as an $A$-module.
    \item If every object of $\mathcal{C}$ is faithfully flat as an $A$-module and $M\neq 0$, then $M$ is faithfully flat as an $A$-module.
    \item If every object of $\mathcal{C}$ is projective as an $A$-module, then $M$ is projective as an $A$-module.
    More precisely, $M$ is a direct sum of objects of $\mathcal{C}$ as an $A$-module.
    \item If every object of $\mathcal{C}$ is free as an $A$-module, then $M$ is free as an $A$-module.
\end{enumerate}
\end{proposition}
\begin{proof}
Taking $U\colon \HM{H}{A} \longrightarrow \M_A $ to be the forgetful functor, which is exact and preserves colimits, Propositions~\ref{ifU} and~\ref{altofilt} imply the statements.
\end{proof}

An object of $\HM{H}{A}$ is said to be \emph{$A$-finite} if it is finitely generated as an $A$-module.

\begin{lemma}\label{containsAfin}
Every nonzero object in $\HM{H}{A}$ has a nonzero $A$-finite subobject.
In particular, every simple object in $\HM{H}{A}$ is $A$-finite.
\end{lemma}
\begin{proof}
Let $0 \neq M\in \HM{H}{A}$. By the fundamental theorem of comodules, $M$ contains a nonzero finite-dimensional $H$-subcomodule $V$.
It is clear that $0 \neq VA$ is a relative $(H,A)$-Hopf submodule of $M$ and is a finitely generated $A$-module.
In particular, when $M$ is a simple object in $\HM{H}{A}$, $M=VA$.
\end{proof}

As an application of the results above, 
we observe that it suffices to consider $A$-finite objects to determine properties of objects of $\HM{H}{A}$ as $A$-modules.
The flat and projective cases appear in \cite[Lemma~7.1]{Sk08} and the free case when $A$ is a Hopf subalgebra of $H$ appears in \cite[Lemma~2]{T79}.
\begin{corollary}
    The following statements hold:
\begin{enumerate}
    \item Every object of $\HM{H}{A}$ is flat as an $A$-module if and only if every $A$-finite object of $\HM{H}{A}$ is flat as an $A$-module.
    \item Every nonzero object of $\HM{H}{A}$ is faithfully flat as an $A$-module if and only if every nonzero $A$-finite object of $\HM{H}{A}$ is faithfully flat as an $A$-module.
    \item Every object of $\HM{H}{A}$ is projective as an $A$-module if and only if every $A$-finite object of $\HM{H}{A}$ is projective as an $A$-module.
    \item Every object of $\HM{H}{A}$ is free as an $A$-module if and only if every $A$-finite object of $\HM{H}{A}$ is free as an $A$-module.
\end{enumerate}
\end{corollary}
\begin{proof}
Put $\mathcal{C} := \{N\in \HM{H}{A} \mid N\neq 0 \text{ and }N \text{ is $A$-finite}\}$.
By Proposition~\ref{semiarttoM}, it suffices to show that every object of $\HM{H}{A}$ is $\mathcal{C}$-semiartinian.
This follows from Lemma~\ref{containsAfin}.
\end{proof}

We next prepare to study the semisimplicity of the category $\HM{H}{A}$.
\begin{proposition}\label{trexprojectiveHM}
Let $M\in\HM{H}{A}$ be $\mathcal{C}$-filtered.
If every object of $\mathcal{C}$ is a projective object in $\HM{H}{A}$, then $M$ is a direct sum of objects of $\mathcal{C}$ in $\HM{H}{A}$.
\end{proposition}
\begin{proof}
The proof is the same as that of Proposition~\ref{trexprojective}.
\end{proof}

\begin{lemma}\label{Afinhasamaximal}
Every nonzero $A$-finite object in $\HM{H}{A}$ has a maximal proper subobject in $\HM{H}{A}$.
\end{lemma}
\begin{proof}
We prove this by Zorn's lemma.
Let $M\in\HM{H}{A}$ be a nonzero $A$-finite object.
Denote by $\Omega$ the set of all proper subobjects of $M$. Since $M$ is nonzero, $0$ is a proper subobject of $M$. Thus $\Omega\neq \emptyset$.
Let $\mathfrak{X}\neq \emptyset$ be a chain in $\Omega$. 
Then $S := \bigcup \mathfrak{X}$ is a subobject of $M$ since a union of an ascending chain of relative $(H,A)$-Hopf submodules is again a relative $(H,A)$-Hopf submodule.
We show $S$ is properly contained in $M$. Assume $M=S$.
Since $M$ is $A$-finite, we can write $M=\sum^n_{i=1} x_iA$ for some $n\in\mathbb{N}, x_i \in M$. 
Taking $X_i \in \mathfrak{X}$ such that $x_i\in X_i$ for $i=1,\ldots,n$, we have $M=\sum^n_{i=1}X_i$.
We can take some $X_j$ containing all $X_i$  since every finite chain has a maximum element.
Thus, $M=\sum^n_{i=1}X_i=X_j$ which contradicts that $X_j$ is a proper subobject of $M$.
Therefore, $S$ is a proper subobject of $M$, and hence $S$ is an upper bound of $\mathfrak{X}$.
By Zorn's lemma, $\Omega$ has a maximal element.
\end{proof}

For an abelian category $\mathscr{A}$, we call $\mathscr{A}$ a \emph{semisimple category} if every object in $\mathscr{A}$ is isomorphic to a direct sum of simple objects in $\mathscr{A}$.
\begin{proposition}\label{semisimpleHM}
The following are equivalent:
\begin{itemize}
    \item [\textup{(i)}] $\HM{H}{A}$ is a semisimple category.
    \item [\textup{(ii)}] Every $A$-finite object in $\HM{H}{A}$ is a projective object in $\HM{H}{A}$.
    \item [\textup{(iii)}] Every simple object in $\HM{H}{A}$ is a projective object in $\HM{H}{A}$.
\end{itemize}
\end{proposition}
\begin{proof}
The implication (i)$\Rightarrow$(ii) is immediate. The implication (ii)$\Rightarrow $(iii) follows from Lemma~\ref{containsAfin}.
We show the implication (iii)$\Rightarrow$(i).
Put $\mathcal{C} := \{S\mid \text{$S$ is a simple object in $\HM{H}{A}$}\}$.
Let $0\neq M\in\HM{H}{A}$. $M$ has a nonzero $A$-finite subobject $N$ by Lemma~\ref{containsAfin}.
By Lemma~\ref{Afinhasamaximal}, we can obtain a maximal proper subobject $P$ of $N$.
Thus, $N/P$ is a simple object in $\HM{H}{A}$ by the correspondence theorem.
Since $N/P$ is a projective object in $\HM{H}{A}$ by the assumption, the following exact sequence splits in $\HM{H}{A}$.
\[0\longrightarrow P \longrightarrow N  \longrightarrow N/P\longrightarrow 0.\]
Hence, $N\cong P\oplus N/P$ in $\HM{H}{A}$, which implies $N/P \hookrightarrow N \subset M$.
We have proved that every nonzero object in $\HM{H}{A}$ contains a simple object in $\HM{H}{A}$.
Thus, every object in $\HM{H}{A}$ is $\mathcal{C}$-semiartinian.
Therefore, every object in $\HM{H}{A}$ is a direct sum of simple objects by the assumption, Propositions~\ref{altofilt} and~\ref{trexprojectiveHM}.
\end{proof}
\begin{remark}
The condition in (ii) was studied in \cite{CG4}, and the implication (ii)$\Rightarrow$(i) for ${}_A \mathfrak{M}^H$ was proved there under the assumption that the algebra $A$ is left Noetherian.
\end{remark}
\begin{remark}
The preceding arguments from Definition~\ref{DefCfiltobj} to Proposition~\ref{semisimpleHM} carry over to the Yetter-Drinfeld category ${}^H_H\YD$, and one can also obtain the analogue of Proposition~\ref{semisimpleHM} for ${}^H_H\YD$.
\end{remark}

\section{Relative Hopf modules over right $H$-simple $H$-comodule algebras over a pointed Hopf algebra}\label{freenesssemisimplicity}

Let $H$ be a Hopf algebra and let $A$ be an $H$-comodule algebra.

\subsection{Freeness of relative Hopf modules over a pointed Hopf algebra}
For $g\in G(H)$ and $M\in {}^H\M$,
we denote by $\twg{M}{g}$ the object of ${}^H\M$ whose $H$-coaction is given by
\[{}_g\rho(m) := g\sw{m}{-1}\otimes \sw{m}{0}.\]
Similarly, we denote by $M^{[g]}$ the object of ${}^H\M$ whose $H$-coaction is given by
\[\rho_g(m) := \sw{m}{-1}g\otimes \sw{m}{0}.\]
In particular, if $M\in\HM{H}{A}$ (resp. $M\in {}^H_A\M$),
then $\twg{M}{g} \in \HM{H}{A}$ (resp. $M^{[g]}\in {}^H_A\M$),
whose underlying $A$-module is $M$.

\begin{proposition}\label{subAAgsame}
Let $g\in G(H)$ and let $I$ be a subspace of $A$.
The following are equivalent:
\begin{itemize}
    \item [\textup{(i)}]  $I$ is an $H$-comodule right ideal of $A$.
    \item [\textup{(ii)}]  $I$ is a relative $(H,A)$-Hopf submodule of $\twg{A}{g}\in\HM{H}{A}$.
\end{itemize}
In particular, $A$ is right $H$-simple if and only if $\twg{A}{g}$ is a simple object in $\HM{H}{A}$.
\end{proposition}
\begin{proof}
Since $A$ and $\twg{A}{g}$ have the same underlying $A$-module,
a right ideal of $A$ is just an $A$-submodule of $\twg{A}{g}$.
Moreover, since the map $\varphi \colon H\otimes A \to H\otimes A;\; x\otimes y \mapsto gx\otimes y$ is a linear isomorphism and $\varphi(H\otimes I) = H\otimes I$,
we have $\rho(I)\subset H\otimes I \iff \varphi(\rho(I))\subset \varphi(H\otimes I) \iff {}_g\rho(I) \subset H\otimes I$.
\end{proof}

We define
\[M_g := \{m\in M \mid \rho(m)=g\otimes m\}\]
for $g\in G(H)$ and $M\in {}^H\M$.

\begin{proposition}\label{annmor}
Let $M\in \HM{H}{A}$, let $g\in G(H)$ and let $m\in M_g$. The map $f \colon \twg{A}{g} \to M ;\; a \mapsto ma$ is a morphism in $\HM{H}{A}$.
\end{proposition}
\begin{proof}
The map $f$ is clearly an $A$-module map. It remains to show that $f$ is an $H$-comodule map.
For any $a\in A$, we have $\rho\circ f(a) = \rho(ma) 
= g\sw{a}{-1}\otimes m\sw{a}{0} = (\mathrm{id}\otimes f)\circ {}_g\rho(a)$.
\end{proof}

\begin{lemma}\label{hassimpAg}
    If $H$ is pointed and $A$ is right $H$-simple, then the following hold:
    \begin{itemize}
        \item For every simple object $M\in\HM{H}{A}$, there exists $g\in G(H)$ such that $M\cong \twg{A}{g}$ in $\HM{H}{A}$.
        \item Every nonzero object $M\in\HM{H}{A}$ has a simple subobject in $\HM{H}{A}$.
    \end{itemize}
\end{lemma}
\begin{proof}
Let $0\neq M\in\HM{H}{A}$.
By Lemma~\ref{comodoverpointed}, there exist $g\in G(H)$ and $0 \neq m\in M$ such that $\rho(m) = g\otimes m$.
The morphism $\twg{A}{g}\to M$ given by Proposition~\ref{annmor} is injective.
Indeed, the kernel of the morphism is zero since it is a subobject of $\twg{A}{g}$,
which is simple by Proposition~\ref{subAAgsame}.
Moreover, if $M$ is simple, then the morphism is also surjective. 
\end{proof}

\begin{remark}\label{embedMgDA}
Proposition~\ref{annmor} and Lemma~\ref{hassimpAg} admit stronger formulations.
Put $D := A^{\mathrm{co}H} \cong \HM{H}{A}(A,A)$. Then for $g\in G(H)$ and $M\in \HM{H}{A}$,
\[M_g \otimes_{D} A \longrightarrow M;\; m \otimes_{D} a \longmapsto ma\]
is a morphism in $\HM{H}{A}$, where the coaction of $M_g \otimes_{D} A$ is given by $\rho(m\otimes_D a)= g\sw{a}{-1} \otimes m \otimes_D \sw{a}{0}$.
Moreover, if $A$ is right $H$-simple, then the above map is injective without assuming pointedness.
This follows because the above map can be identified with 
\[\HM{H}{A}(A,N)\otimes_{\HM{H}{A}(A,A)} A \longrightarrow N;\; f \otimes_{\HM{H}{A}(A,A)} a \longmapsto f(a),\] where $N:= \twg{M}{g^{-1}}$,
and this map is injective by the Simplicity criterion (see \cite[Proposition~12.5]{AMT09}).
\end{remark}

\begin{proposition}\label{pointedfilter}
Suppose  $H$ is pointed and $A$ is right $H$-simple.
Put $\mathcal{C} :=  \{ \twg{A}{g} \mid g\in G(H)\}$.
Then every $M\in \HM{H}{A}$ is $\mathcal{C}$-filtered.
\end{proposition}
\begin{proof}
Every object of $\HM{H}{A}$ is $\mathcal{C}$-semiartinian by Lemma~\ref{hassimpAg}.
So the statement follows from Proposition~\ref{altofilt}.
\end{proof}

In particular, the forgetful functor $U\colon \HM{H}{A} \longrightarrow \M_A$ sends a $\{ \twg{A}{g} \mid g\in G(H)\}$-filtered object
to a $\{A\}$-filtered module (see Lemma~\ref{fpf}).
Therefore, the following freeness holds by Corollary~\ref{trexfree}.

\begin{theorem}\label{mainresult}
    If $H$ is pointed and $A$ is right $H$-simple, then every $M\in \HM{H}{A}$ is free as an $A$-module.
\end{theorem}

\begin{corollary}\label{subfree}
    Suppose $H$ is pointed.
    Let $A\subset B$ be $H$-comodule algebras.
    If $A$ is right $H$-simple, then $B$ is free as an $A$-module.
\end{corollary}
\begin{proof}
Since $B$ is an object of $\HM{H}{A}$ via the $A$-action given by multiplication, the statement holds by Theorem~\ref{mainresult}.
\end{proof}

\subsection{Simple objects in $\HM{H}{A}$}

\begin{proposition}\label{eqrightsimplicity}
Consider the following conditions:
\begin{itemize}
    \item [\textup{(i)}] $A$ is left $H$-simple.
    \item [\textup{(ii)}] $A$ is right $H$-simple.
    \item [\textup{(iii)}] $A\neq 0$, and any $0\neq a\in A_g$ is invertible in $A$ for any $g\in G(H)$.
\end{itemize}
Then each of \textup{(i)} and \textup{(ii)} implies \textup{(iii)}.
Moreover, if $H$ is pointed, then the three conditions are equivalent.
\end{proposition}
\begin{proof}
We show the implication (ii)$\Rightarrow$(iii).
Let $g\in G(H)$ and let $0\neq a\in A_g$.
The map $\twg{A}{g} \to A;\; b \mapsto ab$ is an isomorphism in $\HM{H}{A}$ by Propositions~\ref{subAAgsame}, \ref{annmor} and the right $H$-simplicity of $A$.
By surjectivity, $1 = ab$ for some $b\in A$. Moreover, $ba = 1$ by $a(ba) = a = a\cdot 1$ and injectivity. 
The implication (i)$\Rightarrow$(iii) can be shown by the same argument applied to ${A}^{[g]} \to A;\; b \mapsto ba$.

Next, we show the implication (iii)$\Rightarrow$(ii) when $H$ is pointed. Let $I$ be a nonzero $H$-comodule right ideal of $A$.
By Lemma~\ref{comodoverpointed}, there exist $g\in G(H)$ and $0\neq a\in I$ such that $a\in A_g$. Thus, $1=aa^{-1} \in IA = I$.
The implication (iii)$\Rightarrow$(i) follows similarly.
\end{proof}
\begin{remark}
The equivalence of \textup{(i)} and \textup{(ii)} holds even if $H$ is not pointed when $H$ and $A$ are finite-dimensional \cite[Lemma~2.12]{NSS25}.
\end{remark}

\begin{proposition}\label{schurHM}
Let $g\in G(H)$.
There is an isomorphism 
\[\HM{H}{A}(\twg{A}{g}, M) \cong M_g ;\quad \varphi \mapsto \varphi(1)\]
natural in $M\in\HM{H}{A}$.
In particular, for $h\in G(H)$, $\HM{H}{A}(\twg{A}{g}, \twg{A}{h})\cong A_{h^{-1}g}$.
\end{proposition}
\begin{proof}
We show $\varphi(1)\in M_g$ for $\varphi \colon\twg{A}{g} \to M$ in $\HM{H}{A}$.
Since $\varphi\colon\twg{A}{g} \to M$ is an $H$-comodule map and $\rho_A$ is an algebra map, we have
$\rho_M(\varphi(1)) = (\mathrm{id }\otimes \varphi)\circ {}_g\rho_A(1)=g\otimes \varphi(1)$.
The inverse correspondence is given by $M_g \to \HM{H}{A}(\twg{A}{g}, M) ;\; m \mapsto (a\mapsto ma)$, as in Proposition~\ref{annmor}.
Naturality follows from $(f \circ \varphi)(1) = f(\varphi(1))$ for $f\colon M \to N$ and $\varphi\colon \twg{A}{g} \to M$ in $\HM{H}{A}$.

The last assertion follows from $(\twg{A}{h})_g = A_{h^{-1}g}$.
\end{proof}

\begin{proposition}\label{FAsubgroup}
If $A$ is right $H$-simple, then $F(A):= \{g\in G(H)\mid A_g\neq 0\}$ is a subgroup of $G(H)$.
\end{proposition}
\begin{proof}
Since $\rho(1)=1\otimes1$, $e\in F(A)$.
For $g,h\in F(A)$, there exist $0\neq a\in A_g$ and $0\neq b\in A_h$,
which are invertible by Proposition~\ref{eqrightsimplicity}.
Since $\rho$ is an algebra map, $\rho(b^{-1}a) = \rho(b)^{-1}\rho(a) = (h\otimes b)^{-1}(g\otimes a) = (h^{-1}\otimes b^{-1})(g\otimes a)= h^{-1}g\otimes b^{-1}a$.
Therefore, $0 \neq b^{-1}a \in A_{h^{-1}g}$, and hence $h^{-1}g\in F(A)$.
\end{proof}

So when $A$ is right $H$-simple, we can consider $ G(H) / F(A)$, which is the set of left cosets of $F(A)$ in $G(H)$.
\begin{theorem}\label{pHsHMsimpleclassi}
Let $H$ be a pointed Hopf algebra and let $A$ be a right $H$-simple $H$-comodule algebra.
Then the following map gives a one-to-one correspondence.
\[G(H)/ F(A) \longrightarrow \mathrm{Simp}(\HM{H}{A}) ;\quad [g] \mapsto \twg{A}{g},\]
where  $\mathrm{Simp}(\HM{H}{A})$ denotes the isomorphism classes of simple objects in $\HM{H}{A}$.
\end{theorem}
\begin{proof}
    For $g,h\in G(H)$, 
    $g = h \text{ in } G(H)/F(A) \iff h^{-1}g\in F(A) \iff A_{h^{-1}g}\neq 0$.
    This is equivalent to $\twg{A}{g} \cong \twg{A}{h}$ by Propositions~\ref{subAAgsame} and~\ref{schurHM}.
    Thus, the correspondence is well-defined and injective.
    The surjectivity of the correspondence follows from Lemma~\ref{hassimpAg}.
\end{proof}

\begin{theorem}\label{pHsHMsemisimple}
Let $H$ be a pointed Hopf algebra and let $A$ be a right $H$-simple $H$-comodule algebra.
The following are equivalent:
\begin{itemize}
    \item [\textup{(i)}] $\HM{H}{A}$ is a semisimple category.
    \item [\textup{(ii)}] $A$ is a projective object in $\HM{H}{A}$.
    \item [\textup{(iii)}] There is an $H$-comodule map $\varphi \colon H\to A$ with $\varphi(1) = 1$.
    \item [\textup{(iv)}] $A$ is an injective object in ${}^H\M$.
    \item [\textup{(v)}] The structure map $\rho\colon A\to H\otimes A$ splits in $\HM{H}{A}$.
    \item [\textup{(vi)}] $A$ is coflat as an $H$-comodule.
    \item [\textup{(vii)}] The functor $(-)^{\mathrm{co}H}\colon \HM{H}{A} \to \mathrm{Vect}$ is exact.
\end{itemize}
\end{theorem}
\begin{proof}
First, we show (i)$\Leftrightarrow$(ii). 
By Proposition~\ref{semisimpleHM} and Theorem~\ref{pHsHMsimpleclassi},
the condition~(i) is equivalent to the condition that $\twg{A}{g}$ is a projective object in $\HM{H}{A}$ for any $g\in G(H)$.
So it suffices to show that if (ii) holds, then for any $g\in G(H)$, $\twg{A}{g}$ is a projective object in $\HM{H}{A}$.
This follows because the functor $\twg{(-)}{g} \colon \HM{H}{A} \to \HM{H}{A}$ is an isomorphism of categories, whose inverse functor is $\twg{(-)}{g^{-1}}$,
and so this functor preserves projective objects.

The equivalence of conditions (ii)--(vii) follows from \cite[Proposition~1.3]{VZ96}.
Note that the cited proposition assumes that $H$ has a bijective antipode,
which is satisfied by the fact that every pointed Hopf algebra has a bijective antipode (see, for example \cite[Corollary~7.6.4]{Radtext}).
\end{proof}

\subsection{Semisimplicity in the finite-dimensional case}
We recall some notions about injective $H$-comodules (see, for example \cite[Section~3.5]{Radtext} and~\cite[Section~2.4]{DNRtext})
and use them to examine a criterion in Theorem~\ref{pHsHMsemisimple} in the finite-dimensional case.

An $H$-comodule $N$ is called an \emph{essential extension} of $M$
if $M$ is an $H$-subcomodule of $N$ and
for every $H$-subcomodule $L$ of $N$, if $M\cap L= 0$, then $L=0$.
An \emph{injective envelope} of $M\in {}^H\M$ is an essential extension which is injective in ${}^H\M$.
An injective envelope is unique up to isomorphism. 
We denote the injective envelope of $M\in{}^H\M$ by $E(M)$.

For a monomorphism $i\colon M\hookrightarrow  N$, $i(M)\subset N$ is an essential extension
if and only if for any morphism $f\colon N\to X$, $f\circ i$ is a monomorphism implies that $f$ is a monomorphism.
So a category isomorphism ${}^H \M \to {}^H \M$ preserves essential extensions.
Since a category isomorphism also preserves injective objects, a category isomorphism preserves injective envelopes.

\begin{lemma}\label{EMdim}
Let $H$ be a finite-dimensional pointed Hopf algebra. For a finite-dimensional $H$-comodule $M$,
\[\dim E(M) = \dim \soc{M} \frac{\dim H}{\dim \corad{H}}.\]
In particular, $M$ is injective in ${}^H \M$ if and only if \[\dim M = {\dim \soc{M}}\dfrac{\dim H}{\dim \corad{H}}.\]
\end{lemma}
\begin{proof}
First, $\soc{M} \subset M$ is an essential extension by \cite[Corollary~2.4.12]{DNRtext}.
By the transitivity of essential extensions, $\soc{M}\subset E(M)$ is an essential extension.
In particular, $E(M)$ is an injective envelope of $\soc{M}$ since $E(M)$ is injective.
Hence, we have $E(M) = E(\soc{M})$. So it suffices to show $\dim E(\soc{M}) =  \dim \soc{M}  \dim H/\dim \corad{H}$.

We denote by $\Bbbk a_g$ the one-dimensional $H$-comodule such that $\rho(a_g)=g\otimes a_g$ for $g\in G(H)$.
By Lemma~\ref{comodoverpointed},
we can write $\corad{H} \cong \bigoplus_{g\in G(H)}\Bbbk a_g$, $\soc{M} \cong \bigoplus_{g\in G(H)}(\Bbbk a_g)^{m_g}$ as $H$-comodules where $m_g = \dim M_g$.
In particular, $\dim \soc{M} = \sum_{g\in G(H)}m_g$.
By \cite[Theorem~2.4.16]{DNRtext}, we have 
\[H \cong \bigoplus_{g\in G(H)} E(\Bbbk a_g).\]
Since finite direct sums preserve injective envelopes, 
\[E(\soc{M}) \cong \bigoplus_{g\in G(H)}(E(\Bbbk a_g))^{m_g}.\]

The functor $\twg{(-)}{g} \colon {}^{H}\M \to {}^{H}\M$ is an isomorphism of categories, whose inverse functor is $\twg{(-)}{g^{-1}}$.
So this functor preserves injective envelopes.
Moreover, $\Bbbk a_g \cong \twg{(\Bbbk a_e)}{g}$ as an $H$-comodule.
Hence, $E(\Bbbk a_g) \cong \twg{E(\Bbbk a_e)}{g}$ as an $H$-comodule.
\[H \cong \bigoplus_{g\in G(H)} \twg{E(\Bbbk a_e)}{g}, \quad \quad E(\soc{M}) \cong \bigoplus_{g\in G(H)}( \twg{E(\Bbbk a_e)}{g})^{m_g}.\]
Let $d := \dim E(\Bbbk a_e)$.
Then we have $\dim H = d\cdot |G(H)| = d\cdot \dim \corad{H}$ and
$\dim E(\soc{M}) = \sum_{g\in G(H)}m_g \cdot d= \dim \soc{M}  \dim H/\dim \corad{H}$.

The last assertion follows from
\begin{align*}
    \dim M ={\dim \soc{M}}\frac{\dim H}{\dim \corad{H}} &\iff \dim M = \dim E(M) \\
                &\iff M=E(M) \iff  M \text{ is injective}.
\end{align*}
The proof is done.
\end{proof}

\begin{theorem}\label{fincasesemisimple}
Let $H$ be a finite-dimensional pointed Hopf algebra and let $A$ be a finite-dimensional right $H$-simple $H$-comodule algebra.
Then the following are equivalent:
\begin{itemize}
    \item  [\textup{(i)}] \vphantom{$\displaystyle \frac{\dim A}{\dim \soc{A}}$}
            $\HM{H}{A}$ is a semisimple category.
    \item  [\textup{(ii)}] $\displaystyle \frac{\dim A}{\dim \soc{A}} = \frac{\dim H}{\dim \corad{H}}$.
\end{itemize}
\end{theorem}
\begin{proof}
By Theorem~\ref{pHsHMsemisimple}, it suffices to show that the condition (ii) is equivalent to $A$ being injective in ${}^H\M$.
This follows from Lemma~\ref{EMdim}.
\end{proof}

\begin{remark}\label{mainresultrem}
The arguments in this section can be carried out in the same way with $A^{[g]} \in {}^H_A\M$. 
So we can obtain the left $H$-simple versions of Theorems~\ref{mainresult}, \ref{pHsHMsimpleclassi}, \ref{pHsHMsemisimple} and~\ref{fincasesemisimple} for ${}^H_A \M$,
though we do not need to distinguish between left $H$-simplicity and right $H$-simplicity in view of Proposition~\ref{eqrightsimplicity}.
In particular, the left-handed version of the correspondence in Theorem~\ref{pHsHMsimpleclassi} is given by
\[F(A)\backslash G(H) \longrightarrow \mathrm{Simp}({}^{H}_A\M) ;\quad [g] \mapsto A^{[g]},\]
where $F(A)\backslash G(H)$ is the set of right cosets of $F(A)$ in $G(H)$.
\end{remark}

\begin{corollary}
Assume the base field $\Bbbk$ is algebraically closed of characteristic zero. Let $N>1$ be an odd integer and let $q\in\Bbbk$ be a root of unity of order $N$.
Set $H:=u_q(\mathfrak{sl}_2)$, the small quantum group.
Let $A$ be a finite-dimensional right $H$-simple $H$-comodule algebra.
Then the following are equivalent:
\begin{itemize}
    \item [\textup{(i)}]  $\HM{H}{A}$ is a semisimple category.
    \item [\textup{(ii)}]  $A$ is isomorphic as an $H$-comodule algebra to either $\mathscr{A}_3(r;\xi,\zeta)$ or $\mathscr{A}_3(N;\xi,\zeta,\eta)$ for some $\xi,\zeta,\eta \in \Bbbk$ and some positive divisor $r$ of $N$ with $r<N$.
\end{itemize}
Here $\mathscr{A}_3(r;\xi,\zeta)$ and $\mathscr{A}_3(N;\xi,\zeta,\eta)$ are the $H$-comodule algebras given in \cite[Section~5.6]{NSS25}.
\end{corollary}
\begin{proof}
We have $\dim u_q(\mathfrak{sl}_2)/\dim \corad{u_q(\mathfrak{sl}_2)} = N^2$.
Moreover, \cite[Section~5]{NSS25} gives the complete list of right $H$-simple $H$-comodule algebras, 
and a comodule algebra $A$ in the list that satisfies $\dim A/\dim \soc{A}=N^2$ is $\mathscr{A}_3(r;\xi,\zeta)$ or $\mathscr{A}_3(N;\xi,\zeta,\eta)$
for some $\xi,\zeta,\eta \in \Bbbk$ and some positive divisor $r$ of $N$ with $r<N$.
Thus, the statement follows from Theorem~\ref{fincasesemisimple}.
\end{proof}

\subsection{Classical consequences}

\begin{proposition}\label{coidealsimple}
Let $H$ be a pointed Hopf algebra and let $K$ be a left coideal subalgebra of $H$. Then the following are equivalent:
\begin{itemize}
    \item [\textup{(i)}] $K$ is left $H$-simple.
    \item [\textup{(ii)}] $K$ is right $H$-simple.
    \item [\textup{(iii)}]  $G(K):=G(H)\cap K$ is a subgroup of $G(H)$.
\end{itemize}
\end{proposition}
\begin{proof}
Let $K$ be a left coideal subalgebra.
For $g\in G(H)$ and $0\neq a\in K_g$,
we have $a\otimes 1 =(\mathrm{id}\otimes \varepsilon) \circ \Delta(a) = g\otimes \varepsilon(a)$, 
and hence $a = \varepsilon(a)g$. Thus, $K_g=\Bbbk g$.
Therefore, for $g\in G(H)$ and $0\neq a\in K_g$,
$a$ is invertible in $K$ if and only if the inverse of $g$ is in $G(K)$.
Hence, the statement follows from Proposition~\ref{eqrightsimplicity}.
\end{proof}

The classical results are as follows:
For a pointed Hopf algebra $H$,
every relative $(H,K)$-Hopf module is free as a $K$-module when
\begin{itemize}
    \item $K$ is a Hopf subalgebra of $H$ \cite[Proposition~4]{T79},
    \item $K$ is a one-sided coideal subalgebra of $H$ such that $G(K)$ is a subgroup of $G(H)$ \cite[Proposition~1.4]{Ma91}.
\end{itemize}
We can recover the above assertions by Theorem~\ref{mainresult} and Proposition~\ref{coidealsimple}.
In particular, if $K$ is finite-dimensional, $H$ is free over $K$
since $G(K)$ is a finite submonoid of a group and it is automatically a subgroup of $G(H)$.

Note that for finite-dimensional one-sided coideal subalgebras, Skryabin proved freeness for weakly finite Hopf algebras without assuming pointedness \cite[Theorem~6.1]{Sk07}.

In our pointed case, we can also obtain freeness for finite-dimensional comodule subalgebras.
\begin{proposition}\label{fdsubofrightHsimple}
Let $H$ be a pointed Hopf algebra and let $B$ be a right $H$-simple $H$-comodule algebra.
Then every finite-dimensional $H$-comodule subalgebra of $B$ is right $H$-simple.
\end{proposition}
\begin{proof}
Let $A$ be a finite-dimensional $H$-comodule subalgebra of $B$.
For $g\in G(H)$ and $0\neq a\in A_g$, $a$ is invertible in $B$ by Proposition~\ref{eqrightsimplicity}.
So a linear map $A\to A ; \; x\mapsto ax$ is injective, and hence this is surjective since $A$ is finite-dimensional.
Thus, $1 = ax$ for some $x\in A$, and $1 = a^{-1}(ax)a = xa$.
Therefore, the inverse of $a$ is in $A$, which implies $A$ is right $H$-simple by Proposition~\ref{eqrightsimplicity}.
\end{proof}

\begin{corollary}\label{subcomodfree}
Let $H$ be a pointed Hopf algebra and let $B$ be a right $H$-simple $H$-comodule algebra.
For a finite-dimensional $H$-comodule subalgebra $A$ of $B$, $B$ is free as an $A$-module.
\end{corollary}
\begin{proof}
When $A$ is finite-dimensional, $A$ is automatically right $H$-simple by Proposition~\ref{fdsubofrightHsimple}.
Thus, the statement holds by Corollary~\ref{subfree}.
\end{proof}

\section{A characterization of pointed Hopf algebras by freeness}\label{charapointedbyfreeness}
In this section, we show that the freeness in Theorem~\ref{mainresult} actually characterizes the class of pointed Hopf algebras.

Let $H$ be a Hopf algebra and let $0\neq V$ be a finite-dimensional $H$-comodule with basis $(v_i)_i$.
We write $\rho(v_i) = \sum_{j}h_{i j}\otimes v_j$ where $h_{ij} \in H$.
Then $V^*$ is an $H$-comodule whose coaction is given by $\rho(v_i^*) = \sum_j S(h_{ji}) \otimes v_j^*$ where $(v_i^*)_i$ is the dual basis of $(v_i)_i$.
Moreover,
$\ev\colon V\otimes V^* \to \Bbbk ;\; v\otimes f \mapsto f(v)$ and
$\coev \colon \Bbbk \to V^*\otimes V ;\; 1 \mapsto \sum_i v_i^* \otimes v_i$
are $H$-comodule maps.

Then $D := V^*\otimes V$ is an $H$-comodule algebra 
whose multiplication is given by
\[(V^*\otimes V)\otimes (V^*\otimes V) \overset{\mathrm{id}\otimes \ev \otimes \mathrm{id}}{\longrightarrow} V^*\otimes \Bbbk \otimes V \cong V^*\otimes V\]
and whose identity element is given by $\coev(1)$.

For $W\in {}^H\M$,
$V^*\otimes W$ is an object of ${}^H_D\M$ whose $D$-action is given by
\[D\otimes (V^*\otimes W) \overset{\mathrm{id}\otimes \ev \otimes \mathrm{id}}{\longrightarrow} V^*\otimes \Bbbk \otimes W \cong V^*\otimes W.\]
Similarly, $W \otimes V$ is an object of $\HM{H}{D}$ whose $D$-action is given by
\[(W\otimes V)\otimes D \overset{\mathrm{id}\otimes \ev \otimes \mathrm{id}}{\longrightarrow} W\otimes \Bbbk \otimes V \cong W\otimes V.\]
Note, in particular, that $V^* \cong V^*\otimes \Bbbk \in {}^H_D\M$ and $V\cong \Bbbk\otimes V \in \HM{H}{D}$.

Conversely, for $M \in {}^H_D\M$,
$V\otimes_D M$ is an object of ${}^H\M$ whose $H$-coaction is given by $\rho(v\otimes_D m)= \sw{v}{-1}\sw{m}{-1}\otimes \sw{v}{0}\otimes_D\sw{m}{0}$.
Similarly, for $M \in \HM{H}{D}$,
$M\otimes_D V^*$ is an object of ${}^H\M$ whose $H$-coaction is given by $\rho(m\otimes_D f)= \sw{m}{-1}\sw{f}{-1}\otimes \sw{m}{0}\otimes_D\sw{f}{0}$.

The following is a special case of Morita equivalence in monoidal categories (see \cite[Theorem~5.3]{P78}).
We include a direct proof to make the argument self-contained.
\begin{proposition}\label{HMMorita}
The following hold:
\begin{itemize}
\item The functor ${}^H\M \longrightarrow {}^H_D\M ;\; W \longmapsto V^*\otimes W$ is an equivalence of categories.
\item The functor ${}^H\M \longrightarrow \HM{H}{D} ;\; W \longmapsto W\otimes V$ is an equivalence of categories.
\end{itemize}
\end{proposition}
\begin{proof}
We prove the first statement.
We show that ${}^H_D\M  \longrightarrow {}^H\M ;\; M \longmapsto V\otimes_D M$ is an inverse functor.

Consider an $H$-comodule map $\varphi\colon V \otimes_D V^* \to \Bbbk;\; v\otimes_D f \mapsto f(v)$.
This is well-defined because for $d\otimes w \in V^*\otimes V=D$,
$f(v\cdot(d\otimes w)) = f(d(v)w)= d(v)f(w)= ((d\otimes w)\cdot f)(v)$.
Moreover, $\varphi$ is an isomorphism.
Indeed, since $V\neq 0$, we can choose a basis vector $v_j$ of the basis of $V$, and hence the inverse map is given by $\psi \colon \Bbbk \to V\otimes_D V^* ;\; 1\mapsto v_j\otimes_D v_j^*$.
Note that $\psi \circ \varphi=\mathrm{id}$ holds since
$f(v)(v_j \otimes_D v_j^*) = (v\cdot (f\otimes v_j))\otimes_D v_j^*= v\otimes_D ((f\otimes v_j)\cdot  v_j^*) = v \otimes_D v_j^*(v_j)f = v \otimes_D f$.

Thus, for $W \in {}^H\M$,
\[V\otimes_D V^*\otimes W \overset{\varphi \otimes \mathrm{id}}{\longrightarrow} \Bbbk \otimes W \cong W\]
is an isomorphism. Naturality is immediate.

For $M\in {}^H_D\M$,
\[V^*\otimes V\otimes_D M = D\otimes_D M \cong M.\]
Thus, the first statement holds.
The second statement follows similarly.
\end{proof}

\begin{corollary}
For an $H$-comodule $W$, the following are equivalent:
\begin{itemize}
    \item [\textup{(i)}] $W$ is a simple $H$-comodule.
    \item [\textup{(ii)}] $V^*\otimes W$ is a simple object in ${}^H_D\M$.
    \item [\textup{(iii)}] $W\otimes V$ is a simple object in $\HM{H}{D}$.
\end{itemize}
\end{corollary}
\begin{proof}
Since an equivalence of categories preserves simplicity, the statement holds by Proposition~\ref{HMMorita}.
\end{proof}

In particular, taking $W=V$ and $W=V^*$, we have the following corollary.
\begin{corollary}\label{VdVsimple}
For a finite-dimensional $H$-comodule $V$, the following hold:
\begin{itemize}
    \item If $V$ is a simple $H$-comodule, then $V^*\otimes V$ is left $H$-simple.
    \item If $V^*$ is a simple $H$-comodule, then $V^*\otimes V$ is right $H$-simple.
\end{itemize}
\end{corollary}

\begin{lemma}\label{exisimplesimple}
If $H$ has a simple $H$-comodule whose dimension is greater than one,
then there exists a simple $H$-comodule $V$ with $\dim V>1$ such that $V^*$ is a simple $H$-comodule.
\end{lemma}
\begin{proof}
Let $V$ be a simple $H$-comodule which has minimum dimension among those of dimension greater than $1$.
Assume $V^*$ is not simple.
Then $V^*$ has a proper simple $H$-subcomodule $W$.
Since $V$ has minimum dimension, $\dim W=1$.
Thus, we can write $W = \Bbbk f$ where $0\neq f\in V^*$, $\rho(f)=g\otimes f$ and $g\in G(H)$.
Consider a linear map $\varphi \colon V\to \Bbbk^{[g^{-1}]};\; v\mapsto \ev(v\otimes f)$, where $\Bbbk^{[g^{-1}]}$ is defined in Section~\ref{freenesssemisimplicity}.
We show $\varphi$ is an $H$-comodule map.
Let $v\in V$.
Since $\ev\colon V\otimes V^*\to \Bbbk$ is an $H$-comodule map,
$\sw{\ev(v\otimes f)}{-1}\otimes \sw{\ev(v\otimes f)}{0} = \sw{v}{-1}\sw{f}{-1}\otimes \ev(\sw{v}{0}\otimes \sw{f}{0})
= \sw{v}{-1}g\otimes \ev(\sw{v}{0}\otimes f)$.
Hence, $\sw{\ev(v\otimes f)}{-1}g^{-1}\otimes \sw{\ev(v\otimes f)}{0}
= \sw{v}{-1}\otimes \ev(\sw{v}{0}\otimes f)$, which implies $\varphi$ is an $H$-comodule map.
Since $V$ is simple and $\varphi\neq 0$, $\Ker \varphi = 0$, and hence $V\subset \Bbbk^{[g^{-1}]}$.
This contradicts $\dim V>1$.
Therefore, $V^*$ is simple.
\end{proof}

\begin{theorem}\label{freencharpointed}
Let $H$ be a Hopf algebra. The following are equivalent:
\begin{itemize}
\item [\textup{(i)}] $H$ is pointed.
\item [\textup{(ii)}] For any left $H$-simple $H$-comodule algebra $A$, every object of ${}^H_A\M$ is free as an $A$-module.
\item [\textup{(iii)}] For any right $H$-simple $H$-comodule algebra $A$, every object of $\HM{H}{A}$ is free as an $A$-module.
\end{itemize}
\end{theorem}
\begin{proof}
The implications (i)$\Rightarrow$(ii) and (i)$\Rightarrow$(iii) follow from Theorem~\ref{mainresult} (see also Remark~\ref{mainresultrem}).

Assume $H$ is not pointed. Then, there exists a simple $H$-comodule $V$ with $\dim V >1$ such that $V^*$ is simple by Lemma~\ref{exisimplesimple}.
By Corollary~\ref{VdVsimple}, $D:=V^*\otimes V$ is left $H$-simple and right $H$-simple.
Moreover, $V^* \in {}^H_D\M$ and $V \in \HM{H}{D}$.
So if (ii) or (iii) holds, then $(\dim V)^2 = \dim D\mid \dim V$, which contradicts $\dim V>1$.
Therefore, each of (ii) and (iii) implies $H$ is pointed.
\end{proof}

\section{Right $H$-simple $H$-comodule algebras over a pointed Hopf algebra}\label{rightHsimplecomodalg}

\subsection{Right $\Bbbk G$-simple $\Bbbk G$-comodule algebras}
We begin with the group algebra case.
The following result follows from the argument in \cite[Section~4]{BZS8} (see also \cite{EO4} for a categorical counterpart).
We record it here in the language of $H$-comodule algebras and include a proof for later use.

\begin{proposition}\label{rightKGsimple}
Let $G$ be a group and let $A$ be a $\Bbbk G$-comodule algebra. Then the following are equivalent:
\begin{itemize}
    \item [\textup{(i)}] $A$ is right $\Bbbk G$-simple and $A^{\mathrm{co}\Bbbk G}=\Bbbk$.
    \item [\textup{(ii)}] $A={}_\psi \Bbbk F$ for some subgroup $F$ of $G$ and some $2$-cocycle $\psi$ of $F$.
\end{itemize}
\end{proposition}
\begin{proof}
First, we show the implication (i)$\Rightarrow$(ii).
Since $A$ is a $\Bbbk G$-comodule, we have $A=\bigoplus_{g\in G} A_g = \bigoplus_{g\in F(A)} A_g$,
where $F(A)$ is defined in Proposition~\ref{FAsubgroup}.
For $g\in F(A)$, $A_e\to A_g; x\mapsto ax$ is a linear isomorphism for some $0\neq a\in A_g$ since $a$ is invertible by Proposition~\ref{eqrightsimplicity}.
Therefore, since $A_e = A^{\mathrm{co}\Bbbk G}=\Bbbk$, we have $A_g \cong \Bbbk$.
Thus we can write $A= \bigoplus_{g\in F(A)}\Bbbk a_g$ with $\rho(a_g)=g\otimes a_g$, $a_g\neq 0$ and $a_e=1$.
Moreover, $F(A)$ is a subgroup of $G$ by Proposition~\ref{FAsubgroup}.
Since $\rho(a_ga_h)=\rho(a_g)\rho(a_h)=gh\otimes a_ga_h$, we have $a_ga_h\in A_{gh}=\Bbbk a_{gh}$.
Therefore, there exists a map $\psi \colon F(A)\times F(A) \to \Bbbk^\times$ such that $a_ga_h = \psi(g,h)a_{gh}$.
By associativity, $(a_ga_h)a_k = a_g(a_ha_k)$, and hence $\psi(g,h)\psi(gh,k)a_{ghk} = \psi(g,hk)\psi(h,k)a_{ghk}$.
Since $a_ea_g = a_g = a_ga_e$, we have $\psi(e,g)a_g = a_g = \psi(g,e)a_g$.
Therefore, $\psi$ is a $2$-cocycle of $F(A)$, and hence $A = {}_\psi \Bbbk (F(A))$.

Next, we show the implication (ii)$\Rightarrow$(i).
Let $g\in G$ and $0\neq a=\sum_{f\in F} r_f f \in ({}_\psi\Bbbk F)_g$.
Then we have $ \sum_{f\in F} r_f f\otimes f = \rho(a) = g\otimes a$.
Therefore, if $g\in F$, then $a=r_g g$, while if $g\notin F$, then $a=0$.
Since $a\neq 0$, we have $g\in F$, $a=r_g g$ and hence $a$ is invertible in ${}_\psi\Bbbk F$ with inverse $r_g^{-1}\psi(g,g^{-1})^{-1}g^{-1}$.
Therefore, by Proposition~\ref{eqrightsimplicity}, ${}_\psi\Bbbk F$ is right $\Bbbk G$-simple.
In addition, considering $g=e$, we have $a=r_ee$, and hence $({}_\psi \Bbbk F)^{\mathrm{co}\Bbbk G}=\Bbbk$.
\end{proof}

The assumption of trivial coinvariants in Proposition~\ref{rightKGsimple} is automatic in the finite-dimensional setting over an algebraically closed field.
In fact, this follows from the following elementary observation.

\begin{proposition}\label{ACFcoin}
Assume the base field $\Bbbk$ is algebraically closed.
Let $H$ be a Hopf algebra and let $A$ be an $H$-comodule algebra.
Suppose that $A^{\mathrm{co}H}$ is finite-dimensional and that $A$ is right $H$-simple.
Then $A^{\mathrm{co}H} = \Bbbk$.
\end{proposition}
\begin{proof}
Let $x \in A^{\mathrm{co}H}$.
The linear map $A^{\mathrm{co}H} \to A^{\mathrm{co}H};\; a\mapsto xa$ has an eigenvalue $\lambda\in\Bbbk$
since $\Bbbk$ is algebraically closed and $A^{\mathrm{co}H}$ is a nonzero finite-dimensional vector space.
Thus, $xa = \lambda a$ for some $0 \neq a\in A^{\mathrm{co}H}$.
By Proposition~\ref{eqrightsimplicity}, $a$ is invertible,
and hence $x = \lambda \in \Bbbk$.
\end{proof}

\subsection{Biproduct constructions}
We recall the construction of the Radford biproduct (see, for example \cite[Theorem~11.6.7, 11.6.9]{Radtext}). 
Let $H$ be a Hopf algebra with antipode $S_H$. We use the Sweedler notation $\Delta_H(h)=h_{(1)}\otimes h_{(2)}$ for the comultiplication of $H$
and $\rho_M(m) = \sw{m}{-1}\otimes \sw{m}{0}$
for the $H$-coaction on $M\in{}^{H}_H\YD$.
Let $B$ be a Hopf algebra object in ${}^H_H\YD$.
We use the Sweedler notation $\Delta_B(b)=\swy{b}{1}\otimes \swy{b}{2}$  for the comultiplication of $B$.
Then the vector space $B\otimes H$ has the following Hopf algebra structure:
\begin{itemize}
    \item (multiplication) \quad $(b\otimes h)(c\otimes g) = b(\sw{h}{1}\cdot c)\otimes \sw{h}{2} g$.
    \item (comultiplication) \quad $\Delta(b\otimes h) = (\swy{b}{1}\otimes \sw{\swy{b}{2}}{-1}\sw{h}{1}) \otimes (\sw{\swy{b}{2}}{0}\otimes \sw{h}{2})$.
\end{itemize}
We denote the Hopf algebra by $B\#H$, which is called the Radford biproduct.

The biproduct construction for a comodule algebra object in ${}^H_H\mathcal{YD}$ 
has also been established in \cite[Proposition~7.2 (1)]{M10}.
We recall it here for later use.
\begin{proposition}\label{Momconstruction}
Let $A$ be a $B$-comodule algebra object in ${}^H_H\YD$ and let $F$ be an $H$-comodule algebra.
Then the vector space $A\otimes F$ has the following $B\#H$-comodule algebra structure:
\begin{itemize}
    \item (multiplication) \quad $(a\otimes f)(b\otimes g) = a (\sw{f}{-1}\cdot b) \otimes \sw{f}{0}g$.
    \item ($B\#H$-coaction) \quad $\rho(a\otimes f) = \swy{a}{-1} \# \sw{\swy{a}{0}}{-1}\sw{f}{-1}  \otimes \sw{\swy{a}{0}}{0}\otimes \sw{f}{0}$.
\end{itemize}
Here we use the Sweedler notation $\delta_A(a)=\swy{a}{-1}\otimes \swy{a}{0}$ for the $B$-coaction on $A$
and $\rho_F(f)=\sw{f}{-1}\otimes \sw{f}{0}$
for the $H$-coaction on $F$. 
\end{proposition}
We also denote the $B\#H$-comodule algebra by $A\#F$.

\subsection{Adjunction for the biproduct of comodule algebras}\label{adjforcomalg}

Let $R$ be a Hopf algebra with comultiplication $\Delta_R$, counit $\varepsilon_R$ and antipode $S_R$. 
We use the Sweedler notation $\Delta_R(r)=r_{(1)}\otimes r_{(2)}$  for the comultiplication of $R$.
Let 
\[\pi\colon R \to H \text{ and } \iota \colon H \to R\]
be Hopf algebra maps satisfying $\pi\circ \iota = \mathrm{id}$. 
We define the coinvariant subspace
\[\coinv{R} := \{r\in R \mid (\mathrm{id}\otimes \pi)\circ \Delta_R(r) = r\otimes 1\},\]
which is the image of the linear map
\[\Pi \colon R \longrightarrow R ;\quad  r \longmapsto r_{(1)} (\iota\circ S_H \circ \pi) (r_{(2)}).\]

Then $\coinv{R}$ is a Hopf algebra object in ${}^H_H\mathcal{YD}$ via the following structure:
\begin{itemize}
    \item ($H$-action) \quad $h.r = \iota(h_{(1)})r\iota(S_H(h_{(2)}))$. 
    \item ($H$-coaction) \quad $(\pi\otimes \mathrm{id})\circ \Delta_R$. 
    \item (multiplication) \quad the multiplication inherited from $R$. 
    \item (comultiplication) \quad $(\Pi\otimes \mathrm{id})\circ \Delta_R$. 
\end{itemize}
It is known that $R$ can be decomposed as the Radford biproduct,
i.e., $R\cong \coinv{R} \# H$ as a Hopf algebra via $r\mapsto \Pi(r_{(1)})\otimes \pi(r_{(2)})$ and $b\iota(h)\mapsfrom b\otimes h$ (see, for example \cite[Theorem~11.7.1]{Radtext}).

Since $\coinv{R}$ is a Hopf algebra in ${}^H_H\mathcal{YD}$, we can regard it as an $\coinv{R}$-comodule algebra in ${}^H_H\mathcal{YD}$ with the comultiplication serving as the $\coinv{R}$-coaction.
\begin{proposition}\label{RcFmultcoac}
For an $H$-comodule algebra $F$ with $H$-coaction $\rho_F(\ell) = \sw{\ell}{-1}\otimes \sw{\ell}{0}$ for $\ell\in F$, $\coinv{R}\#F$ has the following $R$-comodule algebra structure:
\begin{itemize}
    \item (multiplication) \quad $(a \otimes \ell)(b\otimes m) = a (\iota(\sw{\ell}{-2}) b \iota(S_H(\sw{\ell}{-1}))) \otimes \sw{\ell}{0}m$.
    \item ($R$-coaction) \quad $\rho(a\otimes \ell)   = \sw{a}{1}\iota(\sw{\ell}{-1}) \otimes \sw{a}{2} \otimes \sw{\ell}{0}$.
\end{itemize}
\end{proposition}
\begin{proof}
By Proposition~\ref{Momconstruction}, $\coinv{R}\#F$ is an $\coinv{R}\# H$-comodule algebra.
The multiplication can be calculated as follows:
\begin{align*}
  (a \otimes \ell)(b\otimes m) &= a (\sw{\ell}{-1}\cdot b) \otimes \sw{\ell}{0}m   \\
                              &= a (\iota(\sw{\sw{\ell}{-1}}{1}) b \iota(S_H(\sw{\sw{\ell}{-1}}{2}))) \otimes \sw{\ell}{0}m  \\
                              &= a (\iota(\sw{\ell}{-2}) b \iota(S_H(\sw{\ell}{-1}))) \otimes \sw{\ell}{0}m 
\end{align*}
by the definition of the $H$-action on $\coinv{R}$.
The $\coinv{R}\# H$-coaction can be calculated as follows:
\begin{align*}
\rho(a\otimes \ell) &= \swy{a}{-1} \# \sw{\swy{a}{0}}{-1}\sw{\ell}{-1}   \otimes \sw{\swy{a}{0}}{0} \otimes \sw{\ell}{0} \\
                    &= \Pi(\sw{a}{1}) \# \pi(\sw{a}{2}) \sw{\ell}{-1} \otimes \sw{a}{3} \otimes \sw{\ell}{0} 
\end{align*}
by the definitions of the $\coinv{R}$-coaction on $\coinv{R}$, inherited from the comultiplication of $\coinv{R}$,
and of the $H$-coaction on $\coinv{R}$.
Via the isomorphism $\coinv{R}\# H  \to R ;\; r\# h \mapsto r\iota(h)$, 
the above $\coinv{R}\# H$-coaction is regarded as the following $R$-coaction:
\begin{align*}
\rho(a\otimes \ell) &= \Pi(\sw{a}{1}) \iota(\pi(\sw{a}{2}) \sw{\ell}{-1}) \otimes \sw{a}{3} \otimes \sw{\ell}{0} \\
                    &= \sw{a}{1}\iota(\sw{\ell}{-1}) \otimes \sw{a}{2} \otimes \sw{\ell}{0}.
\end{align*}
The last equality follows from
\[\Pi(\sw{a}{1}) \iota(\pi(\sw{a}{2})) = \sw{a}{1}S_R(\iota(\pi(\sw{a}{2})))\iota(\pi(\sw{a}{3})) = \sw{a}{1}\varepsilon_R(\sw{a}{2}) = a\]
since $\pi,\iota$ are Hopf algebra maps.
\end{proof}
Thus, we can obtain an $R$-comodule algebra from each $H$-comodule algebra. 
On the other hand, given an $R$-comodule algebra $L$ with coaction $\rho_L$, $L$ is also an $H$-comodule algebra with its original multiplication and $H$-coaction $(\pi\otimes \mathrm{id})\circ \rho_L$.
We write $L^\pi$ for this $H$-comodule algebra.
Then we can consider the following functors
\begin{align*}
\coinv{R}\#(-) 
&\colon {}^H\! \mathrm{Alg} \longrightarrow {}^R\! \mathrm{Alg};\quad
&F&      \longmapsto   \coinv{R}\# F,\quad
&f\colon F\to G  &\longmapsto  \mathrm{id} \otimes f.\\
(-)^\pi 
&\colon {}^R\! \mathrm{Alg} \longrightarrow {}^H\! \mathrm{Alg};\quad
&L&      \longmapsto   L^\pi, \quad
&f\colon L\to M & \longmapsto  f.
\end{align*}
It is a key observation that they form a pair of adjoint functors $(-)^\pi\dashv  \coinv{R}\#(-)$.
To show this,
we recall the notion of the cotensor (see \cite[Section~2.3]{DNRtext}).
For $F\in {}^H\!\mathrm{Alg}$,
the cotensor $R\square_H F := \Ker (((\id_R\otimes \pi)\circ \Delta_R )\otimes \id_F - \id_R \otimes \rho_F)$
is an $R$-comodule algebra whose $R$-coaction is given by $\Delta_R\otimes \id_F$
and whose multiplication is inherited from $R\otimes F$.
The cotensor functor $R\square_H (-) \colon {}^H\M \longrightarrow {}^R\M$ is a right adjoint of $(-)^\pi\colon {}^R\M \longrightarrow {}^H\M$ (see \cite[Proposition~2.3.8]{DNRtext}).
Moreover, in the present setting, this adjunction lifts to an adjunction between the categories of comodule algebras, as follows:
\[\begin{array}{ccc}
    {}^H\!\mathrm{Alg}(L^\pi,F)  &\cong &{}^R\!\mathrm{Alg}(L,R\square_H F); \\   
       f                          &\mapsto& (\id_R \otimes f)\circ \rho_L, \\
       (\varepsilon_R \otimes \mathrm{id}_F)\circ \varphi    & \mapsfrom & \varphi.
    \end{array}
\]
So it is enough to show that the biproduct $\coinv{R} \# F$ is a realization of the cotensor product $R\square_H F$.
\begin{proposition}
The map $\Pi\otimes \mathrm{id} \colon R \square_H F \to \coinv{R} \# F$
is an isomorphism in ${}^R\! \mathrm{Alg}$ and natural in $F\in  {}^H\! \mathrm{Alg}$.
\end{proposition}
\begin{proof}
We prove that the inverse morphism is given by 
\[\psi \colon \coinv{R} \# F \to R \square_H F ;\; a\otimes \ell \mapsto a \iota (\sw{\ell}{-1})\otimes \sw{\ell}{0}.\]
We check that this map indeed takes values in $R \square_H F$.
For $a\in \coinv{R}$ and $\ell\in F$,
\begin{align*}
    ( ((\id\otimes \pi )\circ \Delta_R) \otimes \id)(\psi(a\otimes \ell))
    &= \sw{a}{1} \iota (\sw{\ell}{-2}) \otimes \pi(\sw{a}{2})\pi(\iota(\sw{\ell}{-1})) \otimes  \sw{\ell}{0} \\
    &= a \iota (\sw{\ell}{-2})\otimes \sw{\ell}{-1} \otimes \sw{\ell}{0}\\
    &= (\id \otimes \rho_F )(\psi(a\otimes \ell)).
\end{align*}
Next, we show that $\psi$ is an $R$-comodule algebra map.
For $a,b \in \coinv{R}$ and $\ell,m\in F$,
\begin{align*}
    \psi((a \otimes \ell)(b\otimes m)) 
    &= \psi (a (\iota(\sw{\ell}{-2}) b \iota(S_H(\sw{\ell}{-1}))) \otimes \sw{\ell}{0}m) \\
    &= a\iota(\sw{\ell}{-3}) b \iota(S_H(\sw{\ell}{-2})) \iota(\sw{\ell}{-1})\iota(\sw{m}{-1}) \otimes \sw{\ell}{0}\sw{m}{0} \\
    &= a\iota(\sw{\ell}{-1}) b \iota(\sw{m}{-1}) \otimes \sw{\ell}{0}\sw{m}{0}\\
    &= \psi(a\otimes \ell)\psi(b\otimes m).
\end{align*}
For $a \in \coinv{R}$ and $\ell\in F$,
\begin{align*}
(\id\otimes \psi )\circ \rho(a\otimes \ell)   &= \sw{a}{1}\iota(\sw{\ell}{-1}) \otimes \psi(\sw{a}{2} \otimes \sw{\ell}{0})\\
                                              &= \sw{a}{1}\iota(\sw{\ell}{-2}) \otimes \sw{a}{2}\iota(\sw{\ell}{-1}) \otimes \sw{\ell}{0} \\
                                              &= \Delta_R (a\iota(\sw{\ell}{-1})) \otimes \sw{\ell}{0} = (\Delta_R\otimes \id)\circ \psi(a\otimes \ell).
\end{align*}
It is obvious that $\psi$ preserves the identity element since $\iota$ preserves it.
We show that $\psi$ is the inverse map.
For $a \in \coinv{R}$ and $\ell\in F$,
\begin{align*}
    (\Pi\otimes \id)\circ \psi(a\otimes \ell) &=  \Pi(a\iota(\sw{\ell}{-1}))\otimes \sw{\ell}{0}\\
                                              &= \sw{(a\iota(\sw{\ell}{-1}))}{1} \iota(S_H (\pi(\sw{(a\iota(\sw{\ell}{-1}))}{2})))\otimes \sw{\ell}{0} \\
                                              &= \sw{a}{1}\iota(\sw{\ell}{-2})\iota (S_H(\sw{\ell}{-1})) \iota (S_H (\pi(\sw{a}{2})))\otimes \sw{\ell}{0} \\
                                              &= a\otimes \ell.
\end{align*}
The third equation follows from $a\in \coinv{R}$.
For $\sum_{i}a_i\otimes \ell_i \in R\square_H F$,
\begin{align*}
    \psi\circ (\Pi\otimes \id)(\sum_{i}a_i\otimes \ell_i) &= \sum_i \Pi(a_i)\iota(\sw{\ell_i}{-1})\otimes \sw{\ell_i}{0} \\
                                                          &= \sum_i \Pi(\sw{a_i}{1})\iota(\pi(\sw{a_i}{2}))\otimes \ell_i \\
                                                          &= \sum_i a_i \otimes \ell_i.
\end{align*}
The second equation follows from $\sum_{i}a_i\otimes \ell_i \in R\square_H F$
and the third equation follows from $\Pi(r_{(1)})\iota(\pi(r_{(2)}))=r$ for $r\in R$.

Therefore, $\Pi \otimes \id$ is an isomorphism in ${}^R\! \mathrm{Alg}$. Naturality is obvious.
\end{proof}
Thus we immediately have the following proposition.
\begin{proposition}\label{adcomodalg}
    The following is an isomorphism and is natural in $L\in {}^R\!\mathrm{Alg}$ and $F\in {}^H\!\mathrm{Alg}$.
    \[\begin{array}{ccc}
    {}^H\!\mathrm{Alg}(L^\pi,F)  &\cong &{}^R\!\mathrm{Alg}(L,\coinv{R}\# F); \\   
       f                          &\mapsto& (\Pi \otimes f)\circ \rho_L, \\
       (\varepsilon_R \otimes \mathrm{id}_F)\circ \varphi    & \mapsfrom & \varphi.
    \end{array}
     \]
\end{proposition}
In particular, the correspondence from left to right is useful to examine the biproduct of comodule algebras, so we state it as a lemma.
\begin{lemma}\label{embedsimcomald}
    For an $R$-comodule algebra $L$, an $H$-comodule algebra $F$ and an $H$-comodule algebra map $f\colon L^\pi\to F$, 
    the map $(\Pi\otimes f)\circ \rho_L \colon L \to \coinv{R}\# F$ is an $R$-comodule algebra map.
\end{lemma}

As an example, let us consider the case $H=\Bbbk$ in Proposition~\ref{adcomodalg}.
\begin{example}
   Let $R$ be a finite-dimensional Hopf algebra with unit $\mu\colon \Bbbk\to R$ and counit $\varepsilon\colon R \to \Bbbk$.
   Then $\varepsilon\circ \mu = \mathrm{id}$. So we can apply Proposition~\ref{adcomodalg} with $\pi := \varepsilon$ and $\iota := \mu$.
   We have 
   \[\mathrm{Alg}(L,F) = {}^{\Bbbk}\! \mathrm{Alg}(L^\pi,F)  \cong {}^R\!\mathrm{Alg}(L,\coinv{R}\# F) = {}^R\!\mathrm{Alg}(L,R\otimes F)\] for an $R$-comodule algebra $L$ and an algebra $F$. 
   In the case of $F=\Bbbk$, it means that giving an \emph{augmentation} of $L$, which is an algebra map $L\to \Bbbk$, is the same as giving an $R$-comodule algebra map $L\to R$.
   If $L$ is also right $R$-simple, then $L \to R$ is injective, which implies that $L$ is a left coideal subalgebra of $R$.
   Conversely, every left coideal subalgebra of a finite-dimensional Hopf algebra is right $R$-simple (see \cite[Lemma~3.13]{NSS25}).
   Therefore, we have
   \begin{gather*}
   \{\text{Right $R$-simple augmented comodule algebras}\} \\
    = \{\text{Left coideal subalgebras of }R\},
   \end{gather*}
   where an \emph{augmented} comodule algebra refers to a comodule algebra with an augmentation.
\end{example}

\subsection{Right $H$-simple graded comodule algebras}
First, we recall the definitions of some graded objects.
We call a vector space $V$ a \emph{graded vector space} if it can be written as a direct sum $V=\bigoplus_{n\in \mathbb{N}_0} V(n)$ of subspaces $(V(n))_{n\in\mathbb{N}_0}$ indexed by the set of non-negative integers $\mathbb{N}_0$.
We say that a linear map $f$ from a graded vector space $V=\bigoplus_{n\in \mathbb{N}_0} V(n)$ to a graded vector space $W=\bigoplus_{n\in \mathbb{N}_0} W(n)$ \emph{preserves the grading} if $f(V(n))\subset W(n)$ for $n\in\mathbb{N}_0$.
A \emph{graded Hopf algebra} is a Hopf algebra whose underlying vector space is graded and whose structure maps, namely, multiplication, unit, comultiplication and counit, preserve the grading.
Note that the grading on the base field is defined by $\Bbbk(0)=\Bbbk, \Bbbk(n)=0$ for $n\neq 0$ and that the grading on the tensor product of  $V=\bigoplus_{n\in \mathbb{N}_0} V(n)$ and $W=\bigoplus_{n\in \mathbb{N}_0} W(n)$ is defined by $(V\otimes W)(n) =\bigoplus^n_{i=0}V(n-i)\otimes W(i)$.
Similarly, we can define a \emph{graded comodule algebra} as a comodule algebra whose underlying vector space is graded and whose multiplication, unit and coaction preserve the grading.
We say that a subspace $W$ of a graded vector space $V=\bigoplus_{n\in \mathbb{N}_0} V(n)$ is \emph{graded}  if $W=\sum_{n\in\mathbb{N}_0}(V(n)\cap W)$.
This condition holds if and only if $W$ has the grading $W(n)=V(n)\cap W$ for $n\in \mathbb{N}_0$.

Let $H=\bigoplus_{n\in\mathbb{N}_0}H(n)$ be a graded Hopf algebra.
We denote by $\pi \colon H \to H(0)$ the canonical projection and by $\iota \colon H(0)\to H$ the canonical injection. 
It is clear that $\pi\circ \iota =\mathrm{id}$, so $H$ can be decomposed as $H\cong \coinv{H}\# H(0)$.
Let $L=\bigoplus_{n\in\mathbb{N}_0}L(n)$ be a graded $H$-comodule algebra with coaction $\rho_L$.
We denote by $p \colon L \to L(0)$ the canonical projection.
We can regard $H(0)$ and $L(0)$ as graded with trivial gradings given by
\[(H(0))(n) := \begin{cases}
              H(0) & (n=0), \\
              0    & (\text{otherwise}), 
\end{cases}\quad (L(0))(n) := \begin{cases}
              L(0) & (n=0), \\
              0    & (\text{otherwise}). 
\end{cases}\]
\begin{proposition}
   The $H$-comodule algebra $\coinv{H}\# L(0)$ is a graded $H$-comodule algebra
   with grading given by
   \[(\coinv{H}\# L(0))(n) := (H(n)\cap \coinv{H})\# L(0).\]
\end{proposition}
\begin{proof}
First, $\pi\colon H\to H(0), \iota\colon H(0)\to H$ and $p\colon L\to L(0)$ preserve the grading.
Thus, $\Pi\colon H\to H$, defined in Section~\ref{adjforcomalg}, preserves the grading, as it is a composition of maps that preserve the grading.
Since $\coinv{H} = \mathrm{Im}(\Pi)$ is a graded subspace of $H$, $\coinv{H}\otimes L(0)$ is also a graded vector space with grading given by
\[(\coinv{H}\otimes L(0))(n)=\sum^n_{i=0} (\coinv{H}(n-i))\otimes (L(0)(i)) = (H(n)\cap \coinv{H})\otimes L(0).\]
Furthermore,  since the multiplication and the $H$-coaction of $\coinv{H}\# L(0)$ can be written as compositions of maps that preserve the grading,
the proposition follows.
\end{proof}
\begin{lemma}\label{LptoL0}
The projection $p$ is an $H(0)$-comodule algebra map from $L^\pi$ to $L(0)$.
\end{lemma}
\begin{proof}
We show that $(\pi\otimes p)\circ \rho_L = \rho_L\circ p $.
Let $\ell \in L(n)$. We have $\rho_L(\ell) \in \bigoplus^n_{i=0} H(n-i)\otimes L(i)$ since $\rho_L$ preserves the grading.
If $n=0$, then $(\pi\otimes p)\circ \rho_L(\ell) = \rho_L(\ell) = \rho_L \circ p(\ell) $ since $\rho_L(\ell)\in H(0)\otimes L(0)$.
If $n>0$, then either $i>0$ or $n-i>0$ for $n\geq i\geq 0$.
So $(\pi\otimes p)\circ \rho_L(\ell) = 0= \rho_L \circ p(\ell)$.
\end{proof}

\begin{lemma}\label{embedgradesimpcomodalg}
If $L$ is right $H$-simple, then $L$ is a graded $H$-comodule subalgebra of $\coinv{H}\# L(0)$.
\end{lemma}
\begin{proof}
By Lemma~\ref{embedsimcomald} and~\ref{LptoL0}, we obtain an $H$-comodule algebra map $(\Pi\otimes p)\circ \rho_L \colon L \to \coinv{H}\# L(0)$.
Moreover, this map preserves the grading since it is a composition of maps that preserve the grading.
Since $L$ is right $H$-simple, the map is injective.
\end{proof}

\begin{remark}
A special case, where in particular $H(0)$ is cosemisimple, of this statement appeared in \cite[Proposition~4.4]{V26}.
\end{remark}

\begin{lemma}\label{socleisgraded}
$\soc{L}$ is a graded $H$-subcomodule of  $L$.
Moreover, if $L$ is right $H$-simple, then $\soc{L}\subset L(0)$.
\end{lemma}
%
\begin{proof}
We have $\corad{H}\otimes L = \bigoplus_{n\in\mathbb{N}_0} \corad{H}\otimes L(n)$.
By \cite[Proposition~4.4.9]{Radtext}, $\corad{H}\subset H(0)$, and so $\corad{H}\otimes L(n)\subset H(0)\otimes L(n) \subset (H\otimes L)(n)$ for each $n\in\mathbb{N}_0$.
Thus, $\corad{H}\otimes L$ is a graded subspace of $H\otimes L$.
Hence, $\rho^{-1}(\corad{H}\otimes L)$ is a graded subspace of $L$ since $\rho$ preserves the grading.
Moreover, $\soc{L} = \rho^{-1}(\corad{H}\otimes L)$ by Lemma~\ref{socorad}.
Therefore, $\soc{L}$ is a graded $H$-subcomodule of $L$ since $\soc{L}$ is an $H$-subcomodule of $L$.

Suppose $L$ is right $H$-simple.
Let $k\in\mathbb{N}_0$ be such that $\soc{L}\cap L(k) \neq 0$.
We have $\rho(\soc{L})\subset H(0) \otimes \soc{L} = \bigoplus_{n\in\mathbb{N}_0} H(0) \otimes (L(n)\cap \soc{L})$
since $\soc{L} = \rho^{-1}(\corad{H}\otimes L)$, $\corad{H}\subset H(0)$ and $\soc{L}$ is a graded $H$-comodule.
Moreover, $\rho(L(k))\subset  (H\otimes L)(k)$,
and hence $\rho(\soc{L}\cap L(k))\subset H(0)\otimes (L(k)\cap \soc{L})$.
Therefore, $(\soc{L}\cap L(k))L$ is an $H$-comodule right ideal of $L$.
So we have $(\soc{L}\cap L(k))L = L$. 
Then $1\in L=(\soc{L}\cap L(k))L\subset L(k)L \subset \bigoplus_{n\geq k} L(n)$.
If $k>0$, then this is a contradiction.
Therefore, $\soc{L}\cap L(k) = 0$ for all $k>0$.
Thus, $\soc{L} = \bigoplus_{n\in\mathbb{N}_0} (\soc{L}\cap L(n)) = \soc{L}\cap L(0)$, and hence $\soc{L}\subset L(0)$.
\end{proof}

The filtered analogue of the following proposition was presented in \cite[Theorem~4.4]{NSS25} under the finite-dimensional assumption.
In the graded case, by adding a condition, we can remove the finite-dimensional assumption.

\begin{proposition}\label{gradedrightsimple}
The following are equivalent:
\begin{itemize}
    \item [\textup{(i)}] $L$ is right $H$-simple.
    \item [\textup{(ii)}] $L(0)$ is right $H(0)$-simple and $\soc{L}\subset L(0)$.
\end{itemize}
\end{proposition}
\begin{proof}
We prove the implication (i)$\Rightarrow$(ii).
The statement $\soc{L} \subset L(0)$ follows from Lemma~\ref{socleisgraded}.
We show $L(0)$ is right $H(0)$-simple.
Let $0\neq I$ be an $H(0)$-comodule right ideal of $L(0)$.
Then $IL$ is an $H$-comodule right ideal of $L$, and hence $IL = L$ since $L$ is right $H$-simple.
Thus, we have $L(0) = L(0) \cap IL$.
Moreover, $IL = I\bigoplus_{n\in\mathbb{N}_0} L(n) = \sum_{n\in\mathbb{N}_0} I(L(n))$,
and hence $L(0) \cap IL = I(L(0))$ since $I(L(n))\subset L(0)L(n)\subset L(n)$ for $n\in\mathbb{N}_0$.
Since $I$ is a right ideal of $L(0)$, $I(L(0)) =I$, which implies $L(0)=I$.

We prove the implication (ii)$\Rightarrow$(i).
Let $0\neq I$ be an $H$-comodule right ideal of $L$.
By $\soc{L}\subset L(0)$, $\rho(\soc{L})\subset H(0)\otimes \soc{L}$,
which implies $(I\cap \soc{L})L(0)$ is an $H(0)$-comodule right ideal of $L(0)$.
Moreover, by the fundamental theorem of comodules, $I$ has a simple $H$-subcomodule, and hence $I\cap \soc{L}\neq 0$.
Therefore, $(I\cap \soc{L})(L(0))=L(0)$ since $L(0)$ is right $H(0)$-simple.
Thus, $1\in L(0)= (I\cap \soc{L})(L(0)) \subset I L \subset I$ since $I$ is a right ideal.
\end{proof}

\begin{proposition}\label{embedHckF}
Assume $H(0) = \Bbbk G$ for some group $G$.
If $L$ is right $H$-simple and $L(0)^{\mathrm{co}H(0)} =\Bbbk$, then $L$ is a graded $H$-comodule subalgebra of $\coinv{H}\# ({}_\psi\Bbbk F)$ for some subgroup $F$ of $G$ and some $2$-cocycle $\psi$ of $F$.
\end{proposition}
\begin{proof}
By Lemma~\ref{embedgradesimpcomodalg}, we have $L \subset \coinv{H}\# L(0)$.
By Proposition~\ref{gradedrightsimple}, $L(0)$ is a right $H(0)$-simple $H(0)$-comodule algebra.
Therefore, by Proposition~\ref{rightKGsimple}, $L(0)= {}_\psi \Bbbk F$ for some subgroup $F$ of $G$ and some $2$-cocycle $\psi$ of $F$.
\end{proof}

\begin{lemma}\label{HcFrHsimp}
Assume $H(0) = \Bbbk G$ for some group $G$.
For a subgroup $F$ of $G$ and a $2$-cocycle $\psi$ of $F$, the following hold:
\begin{itemize}
    \item $\coinv{H}\# ({}_\psi\Bbbk F)$ is right $H$-simple.
    \item $(\coinv{H}\# ({}_\psi\Bbbk F))^{\mathrm{co}H} = \Bbbk$.
\end{itemize}
\end{lemma}
\begin{proof}
By \cite[Proposition~4.4.9]{Radtext}, we have $\corad{H}\subset H(0)=\Bbbk G\subset \corad{H}$,
and hence $H$ is a pointed Hopf algebra and $G=G(H)$.
Set $X := (\coinv{H}\# ({}_\psi\Bbbk F))$.
Let $g\in G$ and let $0\neq a\in X_g$.
We write $a=\sum_{f\in F} h_f \otimes f$ where $h_f \in \coinv{H}$ and $f\in F$.
Then we have
\[g \otimes \sum_f h_f \otimes f = \rho(a) = \sum_f \sw{(h_f)}{1}f\otimes \sw{(h_f)}{2}\otimes f\]
by Proposition~\ref{RcFmultcoac}.
Applying the counit on the second tensorand, we obtain
\[\sum_f \varepsilon(h_f) g \otimes   f  = \sum_f h_f f\otimes f.\]
Thus, $\varepsilon(h_f) g = h_f f$ for each $f\in F$. 
Since $h_f \in \coinv{H}$, $\pi(h_f) =\varepsilon (h_f)$.
So $\varepsilon(h_f) g = \pi(\varepsilon(h_f) g) = \pi (h_f f) =\pi (h_f)\pi(f) = \varepsilon (h_f)f$.
Thus, $\varepsilon (h_f) = 0$ for $f\neq g$.
Returning to the equality $\varepsilon(h_f) g = h_f f$, we have $h_f = 0$ for $f\neq g$.
Since $0 \neq a=\sum_f h_f \otimes f$, we have $g\in F$, and hence $a = \varepsilon(h_g)\otimes g$.
Thus, $a$ is invertible in $X$.
By Proposition~\ref{eqrightsimplicity}, the first statement holds.
Moreover, considering the case $g=e$, the second statement holds.
\end{proof}

\begin{theorem}\label{structgrrsimp}
Assume $H(0) = \Bbbk G$ for some finite group $G$.
Then the following are equivalent:
\begin{itemize}
    \item [\textup{(i)}] $L$ is right $H$-simple and $L^{\mathrm{co} H} = \Bbbk$.
    \item [\textup{(ii)}] $L$ is a graded $H$-comodule subalgebra of $\coinv{H}\# ({}_\psi\Bbbk F)$ for some subgroup $F$ of $G$ and some $2$-cocycle $\psi$ of $F$.
\end{itemize}
\end{theorem}
\begin{proof}
The implication (i)$\Rightarrow$(ii) follows from Proposition~\ref{embedHckF},
where the assumption $L(0)^{\mathrm{co}H(0)} =\Bbbk$ in the proposition follows from $L^{\mathrm{co} H} = \Bbbk$.

We prove the implication (ii)$\Rightarrow$(i).
By Lemma~\ref{HcFrHsimp}, $(\coinv{H}\# ({}_\psi\Bbbk F))^{\mathrm{co}H}=\Bbbk$.
Therefore, we have $L^{\mathrm{co}H}=\Bbbk$ since $L\subset \coinv{H}\# ({}_\psi\Bbbk F)$.
By the same argument as in Lemma~\ref{HcFrHsimp}, $H$ is a pointed Hopf algebra.
Let $g\in G$ and let $0\neq x\in L_g$.
Since $x\in L_g \subset (\coinv{H}\# ({}_\psi\Bbbk F))_g$,
$x$ is invertible in $\coinv{H}\# ({}_\psi\Bbbk F)$ by Lemma~\ref{HcFrHsimp} and Proposition~\ref{eqrightsimplicity}.
In particular, $x^{|G|} \neq 0$.
Moreover, by Lagrange's theorem, we have $\rho(x^{|G|})=\rho(x)^{|G|}=g^{|G|}\otimes x^{|G|} = 1 \otimes x^{|G|}$.
Thus, $x^{|G|}\in L^{\mathrm{co}H}=\Bbbk$.
Furthermore, $x^{{|G|}-1}\in L$, and hence $x$ is invertible in $L$.
Therefore, $L$ is right $H$-simple by Proposition~\ref{eqrightsimplicity}.
\end{proof}

As a result,
for a graded Hopf algebra $H$ with $H(0)=\Bbbk G$ for some finite group $G$,
\begin{align*}
&\{\text{Graded right $H$-simple $H$-comodule algebras with trivial coinvariants}\}\\
&= \begin{array}{c}
                        \left\{
                        \begin{array}{l}
                           \text{Graded $H$-comodule subalgebras of $\coinv{H}\# ({}_\psi\Bbbk F)$,}\\
                           \text{for some subgroup $F$ of $G$ and some $2$-cocycle $\psi$ of $F$}
                        \end{array}\right\}.
\end{array}
\end{align*}

\section{A Lagrange-type theorem}\label{Lagrange}
We recall definitions and some facts about the coradical filtration and the Loewy series (see, for example \cite{AD} and \cite[Section~4]{M10}).
For a Hopf algebra $H$, the \emph{coradical filtration} of $H$ is defined by
$H_0:= \corad{H}$ and $H_n := \Delta^{-1}(H_0\otimes H + H\otimes H_{n-1})$ for $n\geq 1$.
For an $H$-comodule algebra $A$, the \emph{Loewy series} of $A$ is defined by
$A_n := \rho^{-1}(H_n\otimes A)$ for $n\geq 0$.
In particular, $A_0 = \soc{A}$ by Lemma~\ref{socorad}.
If $\corad{H}$ is a Hopf subalgebra of $H$,
then the coradical filtration is a Hopf algebra filtration and the Loewy series is an $H$-comodule algebra filtration.
In this case, consider the associated graded algebras
$(\gr H)(n) :=  H_n/H_{n-1}$, $(\gr A)(n) :=  A_n/A_{n-1}$
for $n\in \mathbb{N}_0$ with convention $H_{-1} =0$ and $A_{-1} =0$.
Then $\gr H = \bigoplus_{n\in \mathbb{N}_0} (\gr H)(n)$ forms a graded Hopf algebra
and also $\gr A = \bigoplus_{n\in \mathbb{N}_0} (\gr A)(n)$ forms a graded $\gr H$-comodule algebra.
It is well known that $\gr H$ is \emph{coradically graded}, which means that the coradical filtration is given by $(\gr H)_n = \bigoplus_{i\leq n} (\gr H)(i)$.
$\gr A$ is also \emph{Loewy graded}, which means that the Loewy series is given by $(\gr A)_n = \bigoplus_{i\leq n} (\gr A) (i)$.
\begin{lemma}\label{socgrA}
In the above setting, $(\gr A)(0)=\soc{\gr A}$.
\end{lemma}
\begin{proof}
Since $\gr H$ is coradically graded and $\gr A$ is Loewy graded, we have
$(\gr A)(0) = (\gr A)_0 = \rho^{-1}((\gr H)_0 \otimes \gr A) = \rho^{-1}(\corad{\gr H} \otimes \gr A)$.
By Lemma~\ref{socorad}, $\rho^{-1}(\corad{\gr H} \otimes \gr A) = \soc{\gr A}$.
\end{proof}

The finite-dimensional case of the following lemma follows from \cite[Corollary~4.5]{M10}.
For pointed Hopf algebras, the finite-dimensional assumption is not required.
\begin{lemma}\label{soclerightHsimple}
Let $H$ be a pointed Hopf algebra and let $A$ be an $H$-comodule algebra.
Then the following are equivalent:
\begin{itemize}
    \item [\textup{(i)}] $A$ is right $H$-simple.
    \item [\textup{(ii)}] $\soc{A}$ is right $\corad{H}$-simple.
    \item [\textup{(iii)}] $\gr A$ is right $\gr H$-simple.
\end{itemize}
\end{lemma}
\begin{proof}
We show (i)$\Leftrightarrow$(ii).
For $g\in G(H)$ and $0  \neq a\in A_g$, $\Bbbk a$ is a simple $H$-comodule,
and hence $A_g\subset \soc{A}$. So $A_g = (\soc{A})_g$.
Moreover, any inverse of $a\in A_g$ is in $A_{g^{-1}}\subset \soc{A}$
and $G(H)=G(\corad{H})$.
Therefore, the conditions (i) and (ii) are equivalent by Proposition~\ref{eqrightsimplicity}.

The equivalence (ii)$\Leftrightarrow$(iii) holds by Proposition~\ref{gradedrightsimple} and Lemma~\ref{socgrA}.
\end{proof}


We now prove the following Lagrange-type theorem for right $H$-simple $H$-comodule algebras over a pointed Hopf algebra $H$.
\begin{theorem}\label{prHLag}
Let $H$ be a finite-dimensional pointed Hopf algebra.
For every right $H$-simple $H$-comodule algebra $A$ with $A^{\mathrm{co}H} = \Bbbk$,
$\dim A <\infty$ and $\dim A\mid \dim H$.
\end{theorem}
\begin{proof}
First, $\gr A$ is right $\gr H$-simple by Lemma~\ref{soclerightHsimple}.
Moreover, we have $((\gr A)(0))^{{\mathrm{co}} (\gr H)(0)}  = A_0^{{\mathrm{co}} H_0} \subset A^{{\mathrm{co}} H} = \Bbbk$.
Then $\gr  A\subset \coinv{(\gr H)}\# ({}_\psi \Bbbk F)$ as a $\gr H$-comodule algebra for some subgroup $F$ of $G(H)$ and some $2$-cocycle $\psi$ of $F$ by Proposition~\ref{embedHckF}.
Thus, $\coinv{(\gr H)}\# ({}_\psi \Bbbk F)$ is a free $\gr A$-module by Corollary~\ref{subfree}.
Then $\dim \gr A \mid \dim (\coinv{(\gr H)}\# ({}_\psi \Bbbk F))$.
Moreover, 
\[\dim (\coinv{(\gr H)}\# ({}_\psi \Bbbk F)) = \dim \coinv{(\gr H)} \cdot |F|\]
divides
\[\dim \coinv{(\gr H)} \cdot |G(H)|  = \dim (\coinv{(\gr H)}\# \Bbbk G(H))\]
by Lagrange's theorem.
We also have $\coinv{(\gr H)}\# \Bbbk G(H)\cong \gr H$.
Therefore, $\dim A = \dim \gr A \mid \dim \gr H = \dim H$.
\end{proof}

\begin{corollary}
Assume the base field $\Bbbk$ is algebraically closed.
Let $H$ be a finite-dimensional pointed Hopf algebra.
For every finite-dimensional right $H$-simple $H$-comodule algebra $A$,
$\dim A\mid \dim H$.
\end{corollary}
\begin{proof}
By Proposition~\ref{ACFcoin} and Theorem~\ref{prHLag}.
\end{proof}

\appendix

\section{Further divisibility results}\label{classidimH/Gprime}

Let $H$ be a finite-dimensional pointed Hopf algebra and let $A$ be a finite-dimensional right $H$-simple $H$-comodule algebra.
Then $\dim \soc{A}$ divides $\dim A$ by Corollary~\ref{subcomodfree}.
So we can consider the following statement.
\begin{proposition}\label{divdiv}
The following divisibility holds:
\[\frac{\dim A}{\dim \soc{A}} \mid \frac{\dim H}{\dim \corad{H}}.\]
\end{proposition}
\begin{proof}
By Lemmas~\ref{embedgradesimpcomodalg} and~\ref{soclerightHsimple},
we have $\gr  A\subset \coinv{(\gr H)}\# \soc{A}$ as a $\gr H$-comodule algebra.
Thus, $\coinv{(\gr H)}\# \soc{A}$ is a free $\gr A$-module by Corollary~\ref{subfree}.
Then we have $\dim \gr A \mid \dim (\coinv{(\gr H)}\# \soc{A})$.
We also have $\coinv{(\gr H)}\# \corad{H}\cong \gr H$. Therefore,
\[\frac{\dim \gr A}{\dim \soc{A}} \mid \dim \coinv{(\gr H)} = \frac{\dim \gr H}{\dim \corad{H}}.\]
Since $\dim A = \dim \gr A$ and $\dim H = \dim \gr H$, the statement holds.
\end{proof}

\begin{corollary}
Suppose $\dim H/\dim \corad{H}$ is a prime number.
Then one of the following holds:
\begin{itemize}
    \item $A = \soc{A}$.
    \item $\HM{H}{A}$ is semisimple and $\gr A\cong \coinv{(\gr H)}\# \soc{A}$.
\end{itemize}
In particular, if $A^{\mathrm{co}H} = \Bbbk$,
then one of the following holds:
\begin{itemize}
    \item $A \cong {}_\psi \Bbbk F$,
    \item $\gr A\cong \coinv{(\gr H)}\# {}_\psi \Bbbk F$,
\end{itemize}
where $F$ is some subgroup of $G(H)$ and $\psi$ is some $2$-cocycle of $F$.
\end{corollary}
\begin{proof}
By the assumption and Proposition~\ref{divdiv},
\[\dim A = \dim \soc{A} \quad \text{or} \quad \frac{\dim A}{\dim \soc{A}} = \frac{\dim H}{\dim \corad{H}}.\]
In the latter case, $\HM{H}{A}$ is semisimple by Theorem~\ref{fincasesemisimple}.
We also have $\dim \gr A =  \dim \coinv{(\gr H)} \cdot \dim\soc{A}$.
By Lemmas~\ref{embedgradesimpcomodalg} and~\ref{soclerightHsimple},
$\gr  A\subset \coinv{(\gr H)}\# \soc{A}$, and hence we have $\gr  A\cong \coinv{(\gr H)}\# \soc{A}$.

If $A^{\mathrm{co}H} = \Bbbk$, then $\soc{A}\cong {}_\psi \Bbbk F$
for some subgroup $F$ of $G(H)$ and some $2$-cocycle $\psi$ of $F$ by Proposition~\ref{rightKGsimple}.
So the last assertion holds.
\end{proof}

\end{document}